\documentclass[11pt]{amsart}
\usepackage{mathtools}
\usepackage{hyperref}
\usepackage[all]{xy}
\newcommand{\ZZ}{\mathbb{Z}}

\newcommand{\TT}{\mathbb{T}} 

\newcommand{\cC}{\mathcal{C}} 
\newcommand{\cE}{\mathcal{E}} 
\newcommand{\cF}{\mathcal{F}} 
\newcommand{\cO}{\mathcal{O}} 
\newcommand{\cP}{\mathcal{P}} 
\newcommand{\cS}{\mathcal{S}} 
\newcommand{\cU}{\mathcal{U}} 
\newcommand{\cV}{\mathcal{V}} 

\newcommand{\schemes}{Sch}
\newcommand{\sch}[1]{{\schemes/{#1}}} 
\newcommand{\schx}{\sch{X}} 
\newcommand{\schy}{\sch{Y}} 

\newcommand{\PP}{\mathbb{P}} 
\newcommand{\Hom}{{\mathrm{Hom}}}

\newcommand{\set}{\mathbf{Set}} 
\newcommand{\vbundle}[1]{{\cP({#1})}} 
\newcommand{\bva}{\mathbf{V}} 
\newcommand{\lbdl}{\mathbf{L}} 
\newcommand{\quadmat}[4]{\left(\begin{matrix}#1&#2\\#3&#4\\\end{matrix}\right)}
\newcommand{\ZAR}[1]{(\sch{#1})_{Zar}} 
\newcommand{\zar}[1]{#1_{zar}} 
\newcommand{\res}[2]{{#1}|_{#2}} 
\newcommand{\bO}{\cO} 
\newcommand{\bOX}[1][]{{\bO_X^{\,#1}}} 
\newcommand{\bOY}[1][]{{\bO_Y^{\,#1}}} 
\newcommand{\bOZ}[1][]{{\bO_Z^{\,#1}}} 
\newcommand{\bOU}[1][]{{\bO_U^{\,#1}}} 
\newcommand{\OX}[1][]{{\cO_X^{\,#1}}} 
\newcommand{\OY}[1][]{{\cO_Y^{\,#1}}} 
\newcommand{\OZ}[1][]{{\cO_Z^{\,#1}}} 
\newcommand{\OU}[1][]{{\cO_U^{\,#1}}} 
\newcommand{\OV}[1][]{{\cO_V^{\,#1}}} 
\newcommand{\Spec}{\mathrm{Spec\,}} 
\newcommand{\Proj}{\mathrm{Proj\,}} 
\newcommand{\noplus}{\:\widetilde\oplus\:} 
\newcommand{\notimes}{\:\widetilde\otimes\:} 
\newcommand{\smap}[2]{{#1}_{[#2]}} 

\title{Standard vector bundles}
\author{Youngsoo Kim}
\email{kimy@mytu.tuskegee.edu}
\address{Department of Mathematics \\ John A. Kenny Hall 70-358\\
Tuskegee University\\Tuskegee \\ AL, 36088 \\ USA}


\newcommand{\stvecsp}[1]{\mathcal{V}(#1)}
\newcommand{\FresK}{\res F K}
\newcommand{\pair}[3]{(\{{#1}_{#2}\rightarrow {#1}\},\{{#3}_{#2}\})}

\newtheorem{thm}{Theorem}[section]
\newtheorem{lemma}[thm]{Lemma}
\newtheorem{coro}[thm]{Corollary}
\newtheorem{prop}[thm]{Proposition}
\theoremstyle{definition}
\newtheorem{defn}[thm]{Definition}
\newtheorem{rmk}[thm]{Remark}
\newtheorem{example}[thm]{Example}

\begin{document}

\begin{abstract}
We construct the categories of standard vector bundles
over schemes and define direct sum and tensor product.
These categories are equivalent to the usual categories of vector bundles
with additional properties. The tensor product is strictly associative,
strictly commutative with line bundles, and strictly functorial on base change.
\end{abstract}

\maketitle
\tableofcontents

\section{Introduction}
In a category, two objects $V$ and $W$ could be isomorphic without being equal.
We write $V\cong W$ for an isomorphism and $V=W$ for equality.
For example, if $U,V,$ and $W$ are vector spaces over a field,
then $(U\otimes V)\otimes W$ is isomorphic to $U\otimes(V\otimes W)$
but they are not equal if one defines $\otimes$ in traditional ways. 
Also, $V\otimes W$ is isomorphic to $W\otimes V$, 
but they are not equal in general.
The tensor prodcut of vector spaces is commutative and associative in the sense
that $V\otimes W\cong W\otimes V$ and 
$(U\otimes V)\otimes W\cong U\otimes (V\otimes W)$.
But they are not \emph{strictly commutative} nor \emph{strictly associative} 
in the sense that $V\otimes W \neq W\otimes V$ and 
$(U\otimes V)\otimes W\neq U\otimes (V\otimes W)$.

There are instances we want strictness,
and a typical approach is to construct an equivalent category
and define a new tensor product that is equivalent to the old one.
In fact, it is well known that 
every monoidal category is equivalent to a strict monoidal category \cite[XI.5]{Kassel}.
(See also \cite{Schauenburg}.)
For vector spaces and tensor product, the category can be described as follows.
An object is a finite tuple of vector spaces. 
Each such object corresponds to the tensor product of its components in order.
The homomorphisms between two tuples are homomorphisms between tensored vector spaces.
Then one defines the tensor product of tuples by concatenation.
It is straightforward to verify that this approach defines a category that is equivalent 
to the category of vector spaces, and the new tensor product is strictly associative.
But in this category the tensor product is not strictly commutative. It seems that there is no plausible
way to make both associativity and commutativivity strict. 

The least possible is to make tensor product 
strictly commutative with one-dimensional vector spaces,
keeping strict associativity.
This is achieved with the category of \emph{standard} vector spaces.
Suppose $k$ is a field. For each integer $n\ge0$, 
we call $k^n$ the standard vector space of dimension $n$,
and let $e_1,\ldots, e_n$ be its standard basis.
Consider the category of standard vector spaces over $k$.
$$\stvecsp{k} = \{k^n | n \ge 0\} $$
A homomorphism $h:k^m\rightarrow k^n$ is represented by an $n\times m$ matrix
$M_h$ with respect to the standard bases.
For two standard vector spaces, the tensor product is defined by 
$k^{n_1}\otimes k^{n_2}=k^{n_1n_2}$,
and for homomorphisms $h_1$ and $h_2$, their tensor product $h_1\otimes h_2$ 
is defined to be the homomorphism represented by the matrix $M_{h_1}\otimes M_{h_2}$.
This definition involves a choice in ordering basis of $k^{n_1}\otimes k^{n_2}$.
We choose the order as in
$$e_1\otimes e_1, \ldots, e_1\otimes e_{n_2}, e_2\otimes e_1,\ldots,
e_2\otimes e_{n_2}, \ldots\ldots, e_{n_1}\otimes e_1, \ldots, e_{n_1}\otimes e_{n_2}.$$
Now the tensor product of objects in $\stvecsp{k}$ is strictly associative and
strictly commutative by definition.
The tensor product of homomorphisms is strictly associative because
the tensor product of matrices is associative:
$$(M_{h_1}\otimes M_{h_2})\otimes M_{h_3}=M_{h_1}\otimes(M_{h_2}\otimes M_{h_3}).$$
But it is strictly commutative only if one of the associated matrices is  a $1\times 1$ 
(or empty) matrix.
$$M_{h_1}\otimes M_{h_2}=M_{h_2}\otimes M_{h_1} \quad 
\mbox{if $M_{h_1}$ or $M_{h_2}$ is $1\times 1$}.$$
It is not difficult to see that $\stvecsp{k}$ is equivalent to the usual category
of finite dimensional vector spaces over $k$.

In this article, we construct the categories of standard vector bundles over schemes 
and define tensor product of standard vector bundles.
We prove that these categories are equivalent to the usual categories of 
vector bundles over schemes and that the tensor product 
is strictly associative, strictly commutative with line bundles, and strictly functorial on base change. 
(See Theorem \ref{tp} and Theorem \ref{thmvb} for precise statements.)
The construction uses the above idea of standard vector spaces and the notion of big vector bundles \cite[C.4]{FS} originally from Grayson \cite[p.169]{G} for strict functoriality, 
and the concept of presheaves on sieves is used to combine two ideas.

The standard vector bundles are used in the construction of the motivic symmetric ring spectrum representing algebraic $K$-theory in author's thesis. 
They also solve the question posed in \cite[p.846]{FS}, the existence of strictly functorial tensor product for vector bundles, and their discussion on \emph{small} vector bundles are unnecessary now.
The existence of the standard vector bundles could be used in various $K$-theoretic constructions.

I would like to thank Dan Grayson for calling this problem to my attention and having insightful discussions with me. I also thank Oliver R\"ondigs for reviewing the initial manuscript and making useful comments.

\section{Presheaves defined on a sieve}
In this section, we introduce the notion of presheaves defined on a sieve 
and define sheafification and restriction functors. 
Since the article is about strict equalities, 
everything will be defined concretely, 
not using universal properties. 

In order to avoid set-theoretic problems, we restrict our attention to certain small 
categories of schemes. We let $\schemes$ be the small category of schemes 
that is large enough for one's application 
and that contains all open subschemes of all of its objects.
When we mention a scheme, it will be an object of this category.
Suppose $X$ is a scheme. 
We let $\schx$ denote the category of schemes over $X$.

We begin with the review of Grothendieck topology
and sieves from \cite{SGAIII, FAG} to introduce the notations used throught the article. 
Notations and techniques of proofs follow verbatim those in Chapter 2 of \cite{FAG}.

\subsection{Grothendieck topology and sieves}
Suppose $\TT$ is a small category with all fibered products.
A Grothendieck topology on $\TT$ is an assignment
to each object $U$ of a collection of sets of morphisms $\{U_i\rightarrow U\}_{i\in I}$
called coverings of $U$ such that the following conditions are satisfied. 
We will omit the index set $I$ for simpler notations.
\begin{enumerate}
\item If $V\rightarrow U$ is an isomorphism, then $\{V\rightarrow U\}$ is a covering.
\item If $\{U_i\rightarrow U\}$ is a covering and $V\rightarrow U$ is a morphism, then
$\{U_i\times_U V\rightarrow V\}$ is a covering.
\item If $\{U_i\rightarrow U\}$ is a covering and for all $i$, $\{U_{ij}\rightarrow U_i\}$ is
a covering, then the collection of composites $\{U_{ij}\rightarrow U_i\rightarrow U\}$
is a covering.
\end{enumerate}
A category with a Grothendieck topology is called a {\em site}.
We mainly use Zariski sites on a scheme $X$.
\begin{defn}
The \emph{small Zariski site} $X_{zar}$ on a scheme $X$ is the category whose objects are open immersions $U\rightarrow X$
and morphisms are the open immersions $V\rightarrow U$ compatible with the maps to $X$.
A covering on $U$ is a collection of open immersions $\{f_i:U_i\rightarrow U\}$ such that $\bigcup_i f_i(U_i)=U$.
\end{defn}
\begin{defn}
The \emph{big Zariski site} $\ZAR{X}$ on a scheme $X$
is the category $\schx$ of schemes over $X$ where a
covering on an object $Y$ is a collection of open immersions $\{g_i:V_i\rightarrow Y\}$ such that $\bigcup_i g_i(V_i)=Y$.
\end{defn}

A {\em sieve} on an object $U$ of $\TT$ is a subfunctor
of the representable functor $H_U=\Hom_\TT(-,U)$.
Given a sieve $H$ on $U$,
we can associate a full subcategory $\cC_H$
of the comma category $\TT/U$ over $U$ whose objects
are the elements of $H(V)$ where $V$ runs over the objects of $\TT$.
For simpler notations, when we refer to an object $V\xrightarrow f U$ of $\cC_H$,
we will frequently suppress the structure map and simply write $V$.
No confusion should arise unless two structure maps are considered from the same object.
The category $\cC_H$ satisfies the following property.

\begin{prop}\label{svcatprop}
If $V\xrightarrow f U$ is an object of $\cC_H$
and $g:W\rightarrow V$ is any morphism in $\TT$, then the composite $W\xrightarrow{fg} U$
is also an object of $\cC_H$.
\end{prop}

Conversely, given a full subcategory of $\TT/U$
satisfying the property, we can recover the subfunctor $H$ by defining $H(V)$ to be the
collection of morphisms $V\rightarrow U$ in the category.
Thus we identify a sieve with such a subcategory.
Note that the above property implies that the intersection of two sieves is also a sieve.

Given a collection of morphisms $\cU=\{U_i\rightarrow U\}$,
we associate a sieve $H_\cU$ on $U$ by taking
$$H_\cU(V) = \{f:V\rightarrow U \,|\, \mbox{$f$ factors through $U_i\rightarrow U$ for some $i$}\}$$
If $\TT$ is a site, then a sieve $H$ on $U$ is said to {\em belong to $\TT$} if
$H$ contains a sieve $H_\cU$ associated to some covering $\cU$ of $U$ in $\TT$.
It is equivalent to say that $\cC_H$ contains $\cU$.

A covering $\cV=\{V_j\rightarrow U\}$ is said to be a {\em refinement} of $\cU=\{U_i\rightarrow U\}$ if every map $V_j\rightarrow U$ factors through $U_i\rightarrow U$ for some $i$. The condition is equivalent to $H_\cV \subseteq H_\cU$. If $\cU_1=\{U_{1i}\rightarrow U\}$ and $\cU_2=\{U_{2j}\rightarrow U \}$ are coverings of $U$, then let
$\cU_1\times \cU_2=\{U_{1i}\times_U U_{2j} \rightarrow U\}$.
It is a covering of $U$ and is a common refinement of $\cU_1$ and $\cU_2$.

\begin{prop}[2.44 \cite{FAG}]\label{sievecap}
If $H_1$ and $H_2$ are sieves on $U$ belonging to $\TT$, then the intersection
$H_1\cap H_2$ also belongs to $\TT$.
\end{prop}
\begin{proof}
Let $\cU_1=\{U_{1i}\rightarrow U\}$ and $\cU_2=\{U_{2j}\rightarrow U \}$
be coverings such that
$H_{\cU_1}\subseteq H_1$ and $H_{\cU_2} \subseteq H_2$.
Then $H_1\cap H_2$ contains $H_{\cU_1\times \cU_2}$
\end{proof}

Suppose $f:Y\rightarrow X$ is a map of schemes, and consider big Zariski sites $\ZAR{X}$ and $\ZAR{Y}$. If $V\in \ZAR{Y}$, $U\in\ZAR{X}$ and $g:V\rightarrow U$ is a map of  schemes such that the following diagram commutes,
$$\xymatrix{
V\ar[r]^g \ar[d]_b & U \ar[d]^a\\
Y \ar[r]_f & X
}$$
then for any sieve $H$ on $U$ the pullback $g^*H$ is defined as a sieve on $V$.
For each $W\in \ZAR{Y}$, which is also an object of $\ZAR{X}$ via $f$, the set $(g^*H)(W)$ is defined to be the set of all maps $W\rightarrow V$ such that its composition with $g$ is an element of $H(W)$.

\begin{prop}
Suppose $f:Y\rightarrow X$ is a map of schemes. If $U\xrightarrow a X$ is in $\sch{X}$, $V\xrightarrow b Y$ is in $\schy$, $g:V\rightarrow U$ is a map of schemes such that $ag=fb$, and $H$ is a sieve on $U$ belonging to $\ZAR{X}$, then $g^*H$ is a sieve on $V$ belonging to $\ZAR{Y}$.
\end{prop}
\begin{proof}
Suppose $H$ contains $H_\cU$ where $\cU$ is a Zariski covering of $U$, then $g^*H$ contains $H_{g^*\cU}$ where $g^*\cU=\{U_i\times_U V\rightarrow V\}$, which is a Zariski covering of $V$.
\end{proof}

\subsection{Presheaves and sheaves}
We define presheaves on sieves and construct a sheafifcation functor.
Then various lemmas and formulas needed in section \ref{svb} are developed.
\begin{defn}
Let $X$ be an object of a site $\TT$, and suppose $H$ is a sieve on $X$ belonging to the site.
\begin{enumerate}
\item An {\em $H$-presheaf} is a functor $\cC_H^{op}\rightarrow \set$.
\item An {\em $H$-sheaf} is an $H$-presheaf $F$ such that for each object $U$ of $\cC_H$
and a covering $\{U_i\rightarrow U\}$, the diagram
$$\xymatrix{F(U)\ar[r] & \prod F(U_i)
\ar@<.5ex>[r]^-{p_1^*} \ar@<-.5ex>[r]_-{p_2^*}& \prod F(U_i\times_U U_j)}$$
is exact where $p_1$ and $p_2$ are projections to the first and the second factors of $U_i\times_U U_j$.
\item An $H$-presheaf $F$ is said to be {\em separated} if for each object $U$ of $\cC_H$
and a covering $\{U_i\rightarrow U\}$, the map $F(U)\rightarrow \prod F(U_i)$
is injective.
\end{enumerate}
\end{defn}
Here we use the convention that the value of $F$ on an object $U\xrightarrow f X$ of $\cC_H$ is written as $F(U)$ assuming that the structure map $f$ is understood. When we need to consider two different structure maps $f$ and $g$, we will distinguish them by writing $F(\smap Uf)$ and $F(\smap Ug)$.
By replacing the category of sets with
the category of abelian groups, rings, etc., we get the definitions of $H$-presheaves of abelian groups, rings, etc. A map of $H$-presheaves, ($H$-sheaves, separated $H$-presheaves)
is a natural transformation of functors. 
We denote the categories of $H$-presheaves,
$H$-sheaves, and separated $H$-presheaves by $Pre_H(\TT)$, $Shv_H(\TT)$, and $Pre^s_H(\TT)$, respectively. 
Then $Shv_H(\TT)\subseteq Pre^s_H(\TT)\subseteq Pre_H(\TT)$.
Suppose $H$ and $K$ are sieves belonging to $\TT$ and $K\subseteq H$. Then $\cC_K$ is a
full subcategory of $\cC_H$, and the composition with the inclusion functor induces functors
$Pre_H(\TT)\rightarrow Pre_K(\TT)$, $Shv_H(\TT)\rightarrow Shv_K(\TT)$, and
$Pre_H^s(\TT)\rightarrow Pre_K^s(\TT)$ called restrictions. These functors will be denoted by $-|_K$ universally.
Intuitively, we may consider an $H$-presheaf as a presheaf defined only on {\em small } open sets.
If the site $\TT$ has a terminal object $X$ and $H=\Hom_\TT(-,X)$, the biggest sieve on $X$, then $\cC_H$ is naturally identified with $\TT$.
In this case, the categories are written as $Pre(\TT)$, $Pre^s(\TT)$, and $Shv(\TT)$.
They are identified with the usual categories  of presheaves, separated presheaves, and sheaves.

We can sheafify an $H$-presheaf to obtain a sheaf if $H$ is a sieve on a final object
belonging to a site.
Only local information is needed to define a sheaf after all.
The construction of a sheafification functor $\xi_H:Pre_H(\TT)\rightarrow Shv(\TT)$
presented below follows the construction 
in the proof of Theorem 2.64 in \cite{FAG} of the usual sheafification functor.
First, {\em locally equal} sections are identified to get a separated presheaf, then {\em locally defined} sections are patched together to
obtain a sheaf. The construction works for sheaves of abelian groups, rings, etc., too.

Let $\TT$ be a site and $X$ an object of $\TT$. 
Suppose $H$ is a seive on $X$, and $F$ is an $H$-presheaf.
Then we define an $H$-presheaf $F^s$ by taking $F^s(U)=F(U)/\sim$ 
where we say $s\sim t$ for $s,t\in F(U)$ 
if there is a covering $\{U_i\rightarrow U\}$ 
such that the pullbacks of $s$ and $t$ to each $U_i$ coincide.
We denote the equivalence class of $s\in F(U)$ by $\bar s\in F^s(U)$.
If $f:V\rightarrow U$ is a map in $\cC_H$, 
the pullback $f^*:F(U) \rightarrow F(V)$ is compatible with the equivalence relation, 
so we have a pullback $f^*:F^s(U)\rightarrow F^s(V)$ defined by $\bar s \mapsto \overline{f^*s}$.

\begin{lemma}\label{adjshvsep}
The $H$-presheaf $F^s$ defined above is separated.
Each map $\gamma:F_1\rightarrow F_2$ of $H$-presheaves 
induces a map $\gamma^s:F_1^s\rightarrow F_2^s$.
Thus, we get a functor $Pre_H(\TT)\rightarrow Pre^s_H(\TT)$.
\end{lemma}
\begin{proof}
Suppose $\{U_i\rightarrow U\}$ is a covering of an object $U$ of $\cC_H$.
For $s,t\in F(U)$, if their pullbacks to each $U_i$ coincide, then $s\sim t$ by definition.
Therefore, $F^s(U)\rightarrow \prod F^s(U_i)$ is injective. This proves separatedness.

For each $U$ in $\cC_H$, the map $\gamma(U):F_1(U)\rightarrow F_2(U)$ 
is compatible with the equivalence relation, 
so we have a map $\gamma^s(U):F_1^s(U)\rightarrow F_2^s(U)$ 
defined by $\bar s \mapsto \overline{\gamma s}$.
Since these maps are defined in terms of equivalence classes,
it is straightforward to verify compatibility of various maps.
In particular, the following diagram commutes for any $f:V\rightarrow U$, 
so we get a map $\gamma^s:F_1^s\rightarrow F_2^s$.
$$\xymatrix{
F_1^s(U)\ar[d]_{f^*}\ar[r]^{\gamma^s(U)}&F_2^s(U)\ar[d]^{f^*}\\
F_1^s(V)\ar[r]_{\gamma^s(V)}&F_2(V)\\
}$$

If $\delta:F_2\rightarrow F_3$ is another map of $H$-presheaves, 
then $(\delta\gamma)^s=\delta^s\gamma^s$. Thus, we get a functor $Pre_H(\TT)\rightarrow Pre_H^s(\TT)$.
\end{proof}

Next, we define the sheafification functor $\xi_H:Pre_H(\TT)\rightarrow Shv(\TT)$.
We assume that $X$ is a final object of $\TT$, 
and that $H$ is a sieve on $X$ belonging to $\TT$.
Suppose $U$ is an object of $\TT$.
Consider the set of pairs $(\{U_i\rightarrow U\},\{s_i\})$ 
where $\{U_i\rightarrow U\}$ is a covering of $U$
such that each $U_i$ is in $\cC_H$,  $s_i\in F^s(U_i)$, 
and the pullbacks of $s_i$ and $s_j$ to $U_i\times_U U_j$ coincide.
Note that the set is nonempty since $H$ belongs to $\TT$ and $X$ is a final object,
also that we are free to use pullbacks by Proposition \ref{svcatprop}.
We declare $(\{U_i\rightarrow U\},\{s_i\})$ 
and $(\{V_j\rightarrow U\},\{t_j\})$ are equivalent 
if the pullbacks of $s_i$ and $t_j$ to $U_i\times_U V_j$ coincide. 
The relation is an equivalence relation as will be proved in Lemma \ref{adjshvshff}. 
We define $\xi_H F(U)$ to be the set of equivalence classes $[\{U_i\rightarrow U\},\{s_i\}]$.
Given a map $f:V\rightarrow U$ of $\TT$,
we define $\xi_HF(f):\xi_H F(U)\rightarrow \xi_H F(V)$ 
by sending the class $[\{U_i\rightarrow U\},\{s_i\}]$ 
to the class $[\{U_i\times_U V\rightarrow V\},\{p_i^*s_i\}]$ 
where $p_i^*s_i$ is the pullback of $s_i$ 
along the projection $p_i:U_i\times_U V \rightarrow U_i$.
We also define $\xi_H$ on morphisms. Suppose $\gamma:F_1\rightarrow F_2$ is a map of $H$-presheaves.
For each object $U$ of $\TT$, define $\xi_H\gamma(U):\xi_H F_1(U)\rightarrow \xi_H F_2(U)$ by sending
the class of $(\{U_i\rightarrow U\},\{s_i\})$ to $(\{U_i\rightarrow U\},\{\gamma^s s_i\})$.

\begin{lemma}\label{adjshvshff}
The description in the previous paragraph defines a functor 
$\xi_H:Pre_H(\TT)\rightarrow Shv(\TT)$.
\end{lemma}
\begin{proof}
The relation is reflexive and symmetric by definition. 
To prove that it is transitive, 
suppose $(\{U_i\rightarrow U\},\{s_i\}) \sim (\{V_j\rightarrow U\},\{t_j\})$
and $(\{V_j\rightarrow U\},\{t_j\})\sim(\{W_k\rightarrow U\},\{u_k\})$.
The pullbacks of $s_i$ and $t_j$ to $U_i\times_U V_j$ coincide,
and the pullbacks of $t_j$ and $u_k$ to $V_j\times_U W_k$ coincide.
Then the pullbacks of $s_i$, $t_j$, and $u_k$ to $U_i\times_U\times V_j \times_U W_k$ coincide. 
Since $F^s$ is separated, the pullbacks of $s_i$ and $u_k$ to $U_i\times_U W_k$ coincide.

For the remainder of the proof, we will frequently use the fact that $(\{U_i\rightarrow U\},\{s_i\})$ is equivalent to $(\{V_i\rightarrow U\},\{t_i\})$ if there is an isomorphism $f_i:V_i\rightarrow U_i$ over $U$ for each $i$ and $t_i=f_i^*s_i$.

We prove that the map $\xi_HF(f)$ is well-defined.
First, $U_i\times_U V$ is in $\cC_H$ by Proposition \ref{svcatprop}.
Second, the definition of $\xi_HF(f)$ does not depend on representatives because 
if $(\{U_i\rightarrow U\},\{s_i\})$ and $(\{V_j\rightarrow U\},\{t_j\})$ are equivalent,
then the pullbacks of $s_i$ and $t_j$ coincide in $F^s(U_i\times_U V_j)$ so that 
their pullbacks coincide in $F^s\big((U_i\times_U V)\times_V (V_j\times_U V)\big)
\cong F^s\big((U_i\times_U V_j)\times_U V\big)$.

If $f:V\rightarrow U$ and $g:W\rightarrow V$ are maps in $\TT$,
then $\xi_H F(fg)=\xi_H F(g) \xi_H F(f)$
because if we let $q_i$ be the projection $(U_i\times_U V)\times_V W\rightarrow U_i\times_U V$,
and $r_i$ the projection $U_i\times_U W \rightarrow U_i$,
then the pairs $(\{(U_i\times_U V)\times_V W\rightarrow W\},\{q_i^*p_i^*s_i\})$
and $(\{U_i\times_U W\rightarrow W\},\{r_i^*s_i\})$ are equivalent.
This proves that $\xi_H F$ is a presheaf on $\TT$.

Now we show that $\xi_H F$ satisfies the sheaf conditions.
Let $\{U_i\rightarrow U\}_{i\in I}$ be a covering.
Consider the following sections:
$$([\sigma_i])_{i\in I}=([\{U_{ik} \rightarrow U_i\}_{k\in K_i},\{s_{ik}\}_{k\in K_i}])_{i\in I}\in\prod_{i\in I}\xi_H F(U_i).$$
Assume that the pullbacks of $[\sigma_i]$ and $[\sigma_j]$ coincide in $\xi_H F(U_i\times_U U_j)$,
which means that for each $i,j\in I$,
$(\{U_{ik}\times_U U_j\rightarrow U_i\times_U U_j\}, \{p_{ik}^*s_{ik}\})$
is equivalent to $(\{U_i\times_U U_{jl}\rightarrow U_i\times_U U_j\}, \{q_{jl}^*s_{jl}\})$
where $p_{ik}$ and $q_{jl}$ are projections
$U_{ik}\times_U U_j\rightarrow U_{ik}$ and $U_i\times_U U_{jl}\rightarrow U_{jl}$, respectively.
Then the pullbacks of $s_{ik}$ and $s_{jl}$ along the projections coincide in $F^s(U_{ik}\times_U U_{jl})$ for all $i,j\in I,k\in K_i$, and $l\in K_j$
since $U_{ik}\times_U U_{jl} \cong (U_{ik}\times_U U_j) \times_{(U_i\times_U U_j)} (U_i\times_U U_{jl})$.
Therefore, the pair $\sigma=(\{U_{ik}\rightarrow U\}_{i\in I, k\in K_i},\{s_{ik}\}_{i\in I, k\in K_i})$
defines a section in $\xi_H F(U)$.
We will show that the pullback of $[\sigma]$ to each $U_j$ is $[\sigma_j]$.
The pullback of $[\sigma]$ in $\xi_H F(U_j)$ is the class of the pair
$(\{U_{ik}\times_U U_j\rightarrow U_j\},\{p_{ik}^*s_{ik}\})$.
This pair is equivalent to the pair $\sigma_j=(\{U_{jl}\rightarrow U_j\},\{s_{jl}\})$
because the pullbacks of $p_{ik}^*s_{ik}$ and $s_{jl}$ coincide in
$F^s(U_{ik}\times_U U_{jl})\cong F^s((U_{ik}\times_U U_j) \times_{U_j} U_{jl})$.
This shows the existence of a section.
For uniqueness, suppose $\tau=(\{V_j\rightarrow U\},\{t_j\})$ is another section of $\xi_H F(U)$
whose pullback in $\xi_H F(U_i)$ is equivalent to $\sigma_i$ for all $i$, then
the pullbacks of $t_j$ and $s_{ik}$ coincide in $F^s(V_j\times_U U_{ik})$ for all $i,j$, and $k$.
This implies that $\tau$ is equivalent to $\sigma$.
This completes the proof that $\xi_H F$ is a sheaf.

Next, we verify that $\xi_H\gamma$ is a map of sheaves 
if $\gamma:F_1\rightarrow F_2$ is a map of $H$-presheaves.
If $f:V\rightarrow U$ is a map, the diagram
$$\xymatrix{
\xi_H F_1(U)\ar[r]^{\xi_H\gamma(U)}\ar[d]_{f^*} & \xi_H F_2(U)\ar[d]^{f^*} \\
\xi_H F_1(V)\ar[r]_{\xi_H\gamma(V)} & \xi_H F_2(V) \\
}$$
commutes:
\begin{align*}
f^*\xi_H\gamma(U)[\{U_i\rightarrow U\},\{s_i\}]
&=f^*[\{U_i\rightarrow U\},\{\gamma^s s_i\}]\\
&=[\{U_i\times_U V\rightarrow V\},\{p_i^*\gamma^s s_i\}]\\
&=[\{U_i\times_U V\rightarrow V\},\{\gamma^s p_i^* s_i\}]\\
&=\xi_H\gamma(V)[\{U_i\times_U V\rightarrow V\},\{p_i^* s_i\}]\\
&=\xi_H\gamma(V) f^*[\{U_i\rightarrow U\},\{s_i\}].
\end{align*}
Hence $\xi_H\gamma$ is a map of sheaves. 

If $\delta:F_2\rightarrow F_3$, is another map of $H$-presheaves, then
$\xi_H(\delta\gamma)=(\xi_H\delta)(\xi_H\gamma)$ since $\delta^s\gamma^s=(\delta\gamma)^s$.
Therefore $\xi$ is a functor $Pre_H(\TT)\rightarrow Shv(\TT)$.
\end{proof}

\begin{thm}\label{adjshv}
Let $\TT$ be a site, $X$ a final object of $\TT$, 
$H$ a sieve on $X$ belonging to $\TT$ and
$\eta=-|_H:Shv(\TT)\rightarrow Pre_H(\TT)$ the restriction functor.
Then we can define as above a functor $\xi_H:Pre_H(\TT)\rightarrow Shv(\TT)$
called sheafification,
and there is a natural bijection
$$\Hom_{Shv(\TT)}(\xi_H F, G) \cong \Hom_{Pre_H(\TT)}(F, \eta G)$$
\end{thm}
\begin{proof}
We have defined $\xi_H$ in paragraphs above.
So we will prove that $(\xi_H,\eta)$ is an adjoint pair. 
Suppose $F$ is a $H$-presheaf and $G$ is a sheaf.
Given a map $\alpha:\xi_H F\rightarrow G$ of sheaves,
define a map $\beta_\alpha: F\rightarrow \eta G$ of $H$-presheaves as follows.
If $U$ is an object of $\cC_H$,  $(\{U\xrightarrow1 U\},\bar s)$ defines
a section in $\xi_H F(U)$ for each $s\in F(U)$. 
Define $\beta_\alpha(U):F(U)\rightarrow \eta G(U)$ by sending
$s$ to $\alpha(U)[U\rightarrow U, \bar s]$.
If $f:V\rightarrow U$ is a map of $\cC_H$, then for all $s\in F(U)$,
\begin{align*}
f^*\beta_\alpha(U)(s) 
&= f^*\alpha(U)[\{U\rightarrow U\},\bar s] \\
&= \alpha(U)f^*[U\rightarrow U,\bar s]\\
&= \alpha(V)[U\times_U V \rightarrow V, p^*\bar s]\\
&= \alpha(V)[V\rightarrow V,\overline{f^*s}]\\
&= \beta_\alpha(V)f^*s
\end{align*}
Therefore $\beta_\alpha$ is a map of sheaves.
Conversely, given a map $\beta:F\rightarrow \eta G$ of $H$-presheaves,
define a map $\alpha_\beta:\xi_H F\rightarrow G$ of sheaves as follows.
A section in $\xi_H F(U)$ is a class of a pair $\sigma=(\{U_i\rightarrow U\},\{s_i\})$ such that
$U_i$ is an object of $\cC_H$
and the pullbacks of $s_i$ and $s_j$ to $U_i\times_U U_j$ coincide.
So we have sections $\{\beta(U_i)s_i\}\in\prod \eta G(U_i)$
such that the pullbacks to $U_i\times_U U_j$ of $\beta(U_i)s_i$ and $\beta(U_j)s_j$ coincide.
Since $G$ is a sheaf, there is a unique section in $G(U)$
whose pullback in each $G(U_i)=\eta G(U_i)$ is $\beta(U_i)s_i$.
We call it $\alpha_\beta(U)[\sigma]$.
It does not depend on the representative.
If $\tau=(\{V_j\rightarrow U\}, \{t_j\})$ is equivalent to $\sigma$,
then the pullbacks of $s_i$ and $t_j$ in $F(U_i\times_U V_j)$ coincide so that
the pullbacks of $\alpha_\beta(U)[\sigma]$ and $\alpha_\beta(U)[\tau]$ in $G(U_i\times_U V_j)$ coincide.
Therefore $\alpha_\beta(U)$ is well-defined.
If $f:V\rightarrow U$ is a map in $\TT$,
then for each $\sigma=(\{U_i\rightarrow U\},\{s_i\})$,
the pullbacks of $f^*\alpha_\beta(U)[\sigma]$ and $\alpha_\beta(V)f^*[\sigma]$
in $G(U_i\times_U V)$ for each $i$
are $p_i^*\beta(U_i)s_i=\beta(U_i\times_U V)p_i^*s_i$ where $p_i$ is the projection
$U_i\times_U V \rightarrow U_i$. Therefore $f^*\alpha_\beta = \alpha_\beta f^*$ and
it proves that $\alpha_\beta$ is map of sheaves.
For every object $U\rightarrow X$ in $\cC_H$ and $s\in F(U)$,
$\beta_{\alpha_\beta}(U)s =\alpha_\beta(U)[U\rightarrow U,s] = \beta(U)s$. Hence $\beta_{\alpha_\beta}=\beta$.
For every object $U\rightarrow X$ in $\TT$ and $[\sigma]\in \xi_H F(U)$, $\sigma=(\{U_i\rightarrow U\},\{s_i\})$,
The pullbacks of $\alpha_{\beta_\alpha}(U)[\sigma]$ and $\alpha(U)[\sigma]$ in $G(U_j)$ for every $j$ are
$\beta_\alpha(U_j)(s_j)$ and  $\alpha(U_j)[U_i\times_U U_j\rightarrow U_j, p_i^*s_i]$.
Both are equal to  $\alpha(U_j)[U_j\rightarrow U_j,s_j]$. Therefore $\alpha_{\beta_\alpha} = \alpha$.
This proves the bijection $\Hom_H(\xi_H F,G)\cong \Hom_K(F,\eta G)$.

Finally, to prove that the bijective correspondence is natural, we show that the following diagrams commute
for any $\gamma:F_1\rightarrow F_2$ and $\delta:G_1\rightarrow G_2$.
$$\xymatrix{
\Hom_H(\xi_H F_2,G)\ar[r]^{\beta_\cdot}\ar[d]_{(\xi_H\gamma)^*}&\Hom_K(F_2,\eta G)\ar[d]^{\gamma^*}\\
\Hom_H(\xi_H F_1,G)\ar[r]_{\beta_\cdot} & \Hom_K(F_1,\eta G) \\
 \Hom_H(\xi_H F, G_1) \ar[r]^{\beta_\cdot}\ar[d]_{\delta_*} & \Hom_K(F,\eta G_1)\ar[d]^{(\eta\delta)_*}\\
 \Hom_H(\xi_H F, G_2)\ar[r]_{\beta_\cdot} & \Hom_K(F, \eta G_2)
}$$
If $U$ is an object of $\cC_H$, $s\in F_1(U)$, and $\alpha:\xi_H F_2\rightarrow G$, then
\begin{align*}
\gamma^*\beta_{\alpha}(U)s
&=\beta_{\alpha}(U)(\gamma(U)s)\\
&=\alpha(U)[U\rightarrow U, \overline{\gamma(U)s}],\\
\beta_{(\xi_H \gamma)^*\alpha}(U)s
&=\beta_{\alpha\xi_H\gamma}(U)s\\
&=(\alpha\xi_H\gamma)(U)[U\rightarrow U, \bar s]\\
&=\alpha(U)(\xi_H\gamma)(U)[U\rightarrow U,\bar s]\\
&=\alpha(U)[U \rightarrow U,\overline{\gamma(U)s}].
\end{align*}
So the first diagram commutes.
For the second diagram, let $U$ be an object of  $\cC_H$, $s\in F(U)$,
and $\alpha:\xi_H F\rightarrow G_1$. Then
\begin{align*}
(\eta\delta)_*\beta_\alpha(U)s 
&= (\eta\delta)(U)\beta_\alpha(U)s\\
&=(\eta\delta)(U)\alpha(U)[U\rightarrow U,\bar s]\\
&=\delta(U)\alpha(U)[U\rightarrow U,\bar s]\\
&=(\delta\alpha)(U)[U\rightarrow U,\bar s]\\
&=\beta_{\delta\alpha}(U)s\\
&=\beta_{\delta_*\alpha}(U)s.
\end{align*}

\end{proof}
\begin{lemma}\label{unitmap}
Under the hypothesis of Theorem \ref{adjshv}, the unit map $\epsilon:E\rightarrow (\xi_HE)|_H$ of the adjunction is an isomorphism if $E$ is an $H$-sheaf.
\end{lemma}
\begin{proof}
Since $E$ is separated, $E$ is identified with $E^s$. Using the notation of the proof of Theorem \ref{adjshv},
for each $U\in\cC_H$, $\epsilon(U)$ is defined by $s\mapsto [\{U\xrightarrow 1 U\},s]$.
If $s,t\in E(U)$, and $[\{U\xrightarrow 1 U\}, s] = [\{U\xrightarrow 1 U\}, t]$, then there is a covering $\{U_i\rightarrow U\}$ with each $U_i\in \cC_H$ such that $s|_{U_i} = t|_{U_i}$. Then $s=t$ since $E$ is separated. Hence $\epsilon(U)$ is injective.
For surjectivity of $\epsilon(U)$,
suppose $\sigma=(\{U_i\rightarrow U\},\{s_i\})$ represents an element of $(\xi_H E)(U)$. 
Since $s_i|_{U_{ij}}=s_j|_{U_{ij}}$ for all $i,j$, 
there is an element $s\in E(U)$ such that $s|_{U_i}=s_i$ for all $i$ 
since $E$ is a sheaf. 
Then $(\{U\xrightarrow 1 U\},s)$ and $\sigma$ represent the same element. 
Hence $\epsilon(U)$ is surjective.
\end{proof}

\begin{prop}\label{shffres}
Let $X$ be a final object of a site $\TT$. Suppose $K\subseteq H$ are sieves on $X$ belonging to $\TT$, $F$ is an $H$-presheaf, and $\FresK$ is the restriction of $F$ to $\cC_K$.
Then there is a natural isomorphism $\xi_K (\FresK) \rightarrow \xi_H F$.
\end{prop}
\begin{proof}
We first prove that $(\FresK)^s=F^s|_K$.
For each $U$ in $\cC_K$, 
\begin{align*}
(\FresK)^s(U)&=F|_K(U)/\sim \,\,\,=F(U)/\sim\\
(F^s|_K)(U)&=F^s(U)=F(U)/\sim
\end{align*}
So it is enough to prove that two equivalence relations are the same if $U\in \cC_K$.
If the pullbacks of $s$ and $t$ to each $U_i$ coincide where $\cU=\{U_i\rightarrow U\}$ is a covering
that belongs to $\cC_H$ with $U\in \cC_K$, then there is a refinement $\{U_{ij}\rightarrow U\}$ of $\cU$ that belongs to $\cC_K$, so
that the pullbacks of $s$ and $t$ to each $U_{ij}$ coincide. 

Next, we construct a map $\xi_K(\FresK)\rightarrow \xi_H F$ as follows. Suppose $U$ is an object of $\TT$.
An element of $\xi_K(\FresK)(U)$ is represented by a pair $\sigma=\pair U i s$ such that each $U_i$ is in $\cC_K$ and $s_i\in (\FresK)^s(U_i)=F^s(U_i)$.
The pair also represents an element of $\xi_H F(U)$ since $\cC_K\subseteq \cC_H$.
Also, equivalent representatives of an element of $\xi_K(\FresK)(U)$ represent the same element of $\xi_H F(U)$.
Therefore, we can define $\xi_K(\FresK)(U)\rightarrow \xi_H F(U)$ by sending $[\sigma]$ to $[\sigma]$ (same notation but classes in different equivalence relations).
The definition is compatible with pullbacks, 
so this defines a map of sheaves $\xi_K(\FresK)\rightarrow \xi_H F$. 
From the way it is defined, we see that it is a natural map of sheaves.

Now we prove that it is an isomorphism.
Suppose $\sigma=\pair U i s$ and $\tau=\pair V j t$ represent elements of $\xi_K(\FresK)(U)$ such that $[\sigma]=[\tau]$ in $\xi_H F(U)$.
It implies that the pullbacks of $s_i$ and $t_j$ coincide in $U_i\times_U V_j$.
But $U_i$, $V_j$, and $U_i\times_U V_j$ belong to $\cC_K$. Therefore $\sigma$ and $\tau$ represent the same element of $\xi_K(\FresK)(U)$.
Hence $\xi_K\FresK(U)\rightarrow \xi_H F(U)$ is injective. To prove that it is surjective,
suppose $\sigma=\pair U i s$ represent an element of $\xi_H F(U)$.
Then each $U_i$ is in $\cC_H$. For each $U_i$, there is a covering $\{U_{ij}\rightarrow U_i\}$ such that $U_{ij}\in \cC_K$. 
(For example, $U_{ij}=V_j\times U_i$ 
where $\{V_j\rightarrow X\}$ is a covering of $X$ 
with $V_j\in\cC_K$.)
Then $\{U_{ij}\rightarrow U\}$ is a refinement of $\{U_i\rightarrow U\}$
and the pair $\sigma'=\pair U {ij} s$ where $s_{ij}$ is the pullback of $s_i$ to $U_{ij}$ represent the same element as $\sigma$ does.
But $\sigma'$ also represents an element of $\xi_K (\FresK)(U)$, hence $\xi_K(\FresK)(U)\rightarrow \xi_H F(U)$ is surjective.
\end{proof}

In the big Zariski site $\ZAR{X}$, the big structure sheaf $\OX[b]$ is the sheaf on $\ZAR{X}$ 
that assigns the global sections of $Y$ to each object $Y\xrightarrow f X$.
$$\OX[b](Y)=\OY(Y)$$
For a morphism $g:Z\rightarrow Y$ over $X$, $\OX[b](g):\OX[b](Y)\rightarrow \OX[b](Z)$  is the map of global sections $\OY(Y)\rightarrow \OZ(Z)$ induced by $g$. We simply write $\bOX$ for $\OX[b]$.
If $H$ is a sieve on $X$ belonging to the site, the restriction $\res\bOX H$ is an $H$-sheaf of rings.

From now on, our discussion will be specialized in Zariski topology and presheaves of modules, so an $H$-(pre)sheaf will mean an $H$-(pre)sheaf of $\res\bOX H$-modules unless stated otherwise.
And the notations for categories of $H$-presheaves such as $Pre_H(\ZAR{X})$ will also denote
the categories of presheaves of $\bOX|_H$-modules.

Consider the big Zariski site $\ZAR{X}$ on a scheme $X$. Suppose $H$ is a sieve on $X$ belonging to $\ZAR{X}$ and $F$ is an $H$-presheaf.
For each object $Y\xrightarrow f X$ of $\cC_H$, we define $F|_Y$ to be the restriction of $F$ to the small Zariski site $\zar{Y}$, that is a presheaf on $Y$ in the usual sense.
In other words, $F|_Y(\smap Ug)=F(\smap U{fg})$ for each open immersion $g:U\rightarrow Y$,
and $F|_Y(h)=F(h)$ for each map $h:V\rightarrow U$ of $\zar{Y}$, which may be considered as a map of $\cC_H$.
We will call $F|_Y$ the restriction of $F$ to $Y$ along $f$.
If $G$ is another $H$-presheaf and there is a map of $H$-presheaves $F\rightarrow G$, we get a natural map
$F|_Y\rightarrow G|_Y$.
So the restriction is a functor $Pre_H(\ZAR{X})\rightarrow Pre(\zar{Y})$.
\begin{prop}\label{shffresii}
Suppose $X$ is a scheme, $H$ is a sieve on $X$ belonging to $\ZAR{X}$, and $F$ is an $H$-presheaf.
If $Y$ is an object of $\cC_H$, then $(\xi_H F)|_Y = \xi (F|_Y)$ (equality, not isomorphism) where $\xi_H$ and $\xi$ are sheafification functors of Theorem \ref{adjshv}.
\begin{align*}
\xi_H&:Pre_H(\ZAR{X})\longrightarrow Shv(\ZAR{X})\\
\xi&:Pre(\zar{Y})\longrightarrow Shv(\zar{Y})
\end{align*}
\end{prop}
\begin{proof}
Note that $(F|_Y)^s(U)=F|_Y(U)/\sim = F(U)/\sim = F^s(U)$ 
for any open immersion $U\rightarrow Y$
since a covering is a Zariski covering for both sites $\ZAR{X}$ and $\zar{Y}$.
For each open immersion $U\rightarrow Y$,
$(\xi_H F)|_Y(U)=(\xi_H F)(U)$ is the set of pairs $(\{U_i\rightarrow U\},\{s_i\})$ modulo an equivalence relation
where $\{U_i\rightarrow U\}$ is a Zariski cover, and $s_i\in F^s(U_i)$ for each $i$.
Similarly, $\xi (F|_Y)$ is the set of such pairs with 
$s_i\in (F|_Y)^s(U_i)=F^s(U_i)$ modulo an equivalence relation.
Both of them have the same collection of Zariski covers, and the same equivalence relations. 
Therefore, $(\xi_H F)|_Y(U)=\xi(F|_Y)(U)$.
The following diagram is commutative for any $V\rightarrow U$ by definition.
$$\xymatrix{(\xi_H F)|_Y(U)\ar[r]\ar@{=}[d] & (\xi_H F)|_Y(V)\ar@{=}[d] \\ \xi (F|_Y)(U)\ar[r] &\xi (F|_Y)(V)}$$
This completes the proof.
\end{proof}

Let $H$ be a sieve on $X$ belonging to $\ZAR{X}$, and $Y\in\cC_H$. 
If $F$ is an $H$-sheaf, then $F|_Y$ is a sheaf of $\OY$-modules. 
If $g:Z\rightarrow Y$ is a map in $\cC_H$, then for every open immersion $U\rightarrow Y$,
there is a map
$$F|_Y(U) = F(U)\xrightarrow{F(\pi_U)} F(U\times_Y Z) = F|_Z(U\times_Y Z)=g_*F|_Z(U),$$
and the diagram below induced by a map $V\rightarrow U$ commutes.
$$\xymatrix{
\res FY(U) \ar[r]\ar[d] & g_*\res FZ(U) \ar[d] \\
\res FY(V) \ar[r] & g_*\res FZ(V) \\
}$$
Hence there is a map $\rho_{F,g}:F|_Y \rightarrow g_*F|_Z$.
By adjointness, we get a natural map $\lambda_{F,g}:g^*(F|_Y)\rightarrow F|_Z$ of sheaves of $\OZ$-modules.

We can define an extension of a sheaf from the small to the big Zariski site.
Given a sheaf $\cF$ of $\OX$-modules, define $B\cF$, a sheaf on $\ZAR{X}$ by setting $B\cF(Y)=f^*\cF(Y)$ for each object $Y\xrightarrow f X$ of $\ZAR{X}$. If $g:Z\rightarrow Y$ is a map over $X$, $B\cF(g)$ is defined to be the composite
$$f^*\cF(Y)\rightarrow g^*f^*\cF(Z) \xrightarrow\cong (fg)^*\cF(Z)$$
 induced by the map of global sections. The commutativity of the following diagram shows $B\cF(gh)=B\cF(h)B\cF(g)$ for $g:Z \rightarrow  Y$ and $h:W\rightarrow Z$
 $$\xymatrix{
 f^*\cF(Y)\ar[r]\ar[d] & (gh)^*f^*\cF(W) \ar[d] \ar[ddr] \\
 g^*f^*\cF(Z) \ar[r] \ar[d] & h^*g^*f^*\cF(W) \ar[d] \\
 (fg)^*\cF(Z) \ar[r] & h^*(fg)^*\cF(W)\ar[r]  & (fgh)^*\cF(W)
 }$$
\begin{lemma}\label{extres}
Suppose $\cF$ is a sheaf on $\zar{X}$, and $B\cF$ the extension of $\cF$ to $\ZAR{X}$. Then
\begin{enumerate}
\item for each object $Y\xrightarrow f X$ of $\ZAR{X}$, there is a natural isomorphism $\res{B\cF}Y\rightarrow f^*\cF$, thus $B\cF$ is a sheaf,
\item for each map $g:Z\rightarrow Y$ over $X$, the induced map $\lambda:g^*(\res{B\cF}Y) \rightarrow \res{B\cF}Z$ is an isomorphism.
\end{enumerate}
\end{lemma}
\begin{proof}
For the first statement, suppose $g:U\rightarrow Y$ is an open immersion. Then $\res{B\cF}Y(U)=B\cF(U)=(fg)^*\cF(U)$.
Define $\res{B\cF}Y(U)\rightarrow f^*\cF(U)$ to be the composite
$(fg)^*\cF(U)\xrightarrow\cong g^*f^*\cF(U)\xrightarrow\cong f^*\cF(U)$.
If $h:V\rightarrow U$ is an open immersion, the following diagram commutes.
$$\xymatrix{
(fg)^*\cF(U)\ar[r] \ar[d] & g^*f^*\cF(U) \ar[r]\ar[d] & f^*\cF(U) \ar[dd] \\
h^*(fg)^*\cF(V) \ar[r] \ar[d] &  h^*g^*f^*\cF(V)\ar[d]  \\
(fgh)^*\cF(V)\ar[r]  & (gh)^*f^*\cF(V) \ar[r] & f^*\cF(V)
}$$
All of the maps involved in the diagram are natural in $\cF$.
This proves the first statement. For the second, note that the following diagram commutes.
$$\xymatrix{
g^*(\res{B\cF}Y) \ar[r] \ar[d]_-\cong & \res{B\cF}Z \ar[d]^-\cong \\
g^*f^*\cF \ar[r]_-\cong & (fg)^*\cF
}$$
Three isomorphisms in the diagram implies that the top arrow is an isomorphism.
\end{proof}
Since the definition of $B$ is functorial in $\cF$, we have defined a functor
$$B:Shv(\zar{X})\rightarrow Shv(\ZAR{X}).$$
\begin{lemma}\label{resext}
Suppose $F$ is a sheaf on $\ZAR{X}$ such that the induced map $\lambda_{F,f}:f^*(F|_X)\rightarrow F|_Y$ is an isomorphism for every object $Y\xrightarrow f X$ of $\ZAR{X}$.
Then there is an isomorphism $\eta:B(F|_X) \rightarrow F$ that is natural in the sense that if $G$ is another such sheaf, and there is a map $\alpha:F\rightarrow G$, then
the following diagram commutes.
$$\xymatrix{
B(F|_X) \ar[r]^-\eta \ar[d]_{B\alpha|_X} &
F \ar[d]^\alpha\\
B(G|_X) \ar[r]_-\eta &
G
}$$
\end{lemma}
\begin{proof}
For each object $Y\xrightarrow f X$, define $\eta(Y)=\lambda_{F,f}(Y)$.
$$B(F|_X)(Y) = f^*(F|_X)(Y) \xrightarrow[\cong]{\lambda_{F,f}(Y)} F|_Y(Y) = F(Y)$$
To show that $\eta$ is an isomorphism of functors, we need to show the commutativity of the following diagram for
every map $g:Z\rightarrow Y$ over $X$.
$$\xymatrix{
B(F|_X)(Y) \ar[r] \ar[d] &
F(Y) \ar[d] \\
B(F|_X)(Z) \ar[r] &
F(Z)
}$$
It is enough to show the commutativity of the following diagram.
$$\xymatrix{
f^*(F|_X)(Y) \ar[rr]^{\lambda_f(Y)} \ar[d]&&
F|_Y(Y) \ar[d] \ar@/^2cm/[dd]^{F(g)}\\
g^*f^*(F|_X)(Z) \ar[rr]^{g^*\lambda_f(Z)} \ar[d] &&
g^*(F|_Y)(Z) \ar[d]_{\lambda_g(Z)} \\
(fg)^*(F|_X)(Z) \ar[rr]_{\lambda_{fg}(Z)} &&
F|_Z(Z)
}$$
The top square is commutative since it is induced by the natural transformation $1\rightarrow g_*g^*$.
The bottom square is commutative since the corresponding diagram of sheaves before taking the global sections commutes.
On the level of stalks, it corresponds to the diagram of modules
$$\xymatrix {
C\otimes_B(B\otimes_A M) \ar[r] \ar[d] & C\otimes_A N \ar[d]\\
C\otimes_A M \ar[r] &L
}$$
induced by rings $A,B,C$, an $A$-module $M$, a $B$-module $N$, and a $C$-module $L$, together with
morphisms of rings $A\rightarrow B\rightarrow C$, a $B$-linear map $M\rightarrow N$, and a $C$-linear map $N\rightarrow L$.
Finally, the part on the right is obtained by taking the global sections of the following diagram of sheaves,
$$\xymatrix{
F|_Y \ar[d] \ar[dr]^{\rho_g}\\
g_*g^*F|_Y \ar[r]_{g_*\lambda_g} &
g_*F|_Z
}$$
which is commutative since $\rho$ and $\lambda$ corresponds to each other in the adjoint relationship of $g_*$ and $g^*$.
The naturality of $\eta$ follows from the naturality of $\lambda_{F,f}$.
\end{proof}

Let $f:Y\rightarrow X$ be a map of schemes and $H$ a sieve on $X$ belonging to $\ZAR{X}$. 
We will define a pullback functor 
$$f^*:Pre_H(\ZAR{X}) \rightarrow Pre_{f^*H}(\ZAR{Y}).$$
Recall that $f^*H$ is the sieve on $Y$ belonging to $\ZAR{Y}$ such that $Z\xrightarrow g Y$ is an object of $\cC_{f^*H}$ if and only if the composition $Z\xrightarrow g Y \xrightarrow f X$ is in $\cC_H$.
Therefore, there is a functor $f_*:\cC_{f^*H} \rightarrow \cC_H$ defined by composition with $f$.
Then for an $H$-presheaf $F$, $f^*F$ is defined to be $Ff_*^{op}$, that is, $f^*F(\smap Zg)= F(\smap Z{fg})$ for each object $Z\xrightarrow g Y$ of $\cC_{f^*H}$ and $f^*F(h) = F(h)$ for each morphism $h$ of $\cC_{f^*H}$, which may be considered as a morphism of $\cC_H$ as well. In addition, for a map $\alpha:E\rightarrow F$ of $H$-presheaves, $f^*\alpha:f^*E\rightarrow f^*F$ is defined by $(f^*\alpha)(\smap Zg) = \alpha(\smap Z{fg})$ for each object $Z\xrightarrow g Y$ of $\cC_{f^*H}$. The following diagram commutes for any morphism $h$ of $\cC_H$, thus $f^*\alpha$ is indeed a map of $f^*H$-presheaves.
$$\xymatrix{
f^*E(Z) \ar@{=}[r] \ar[d]_{f^*E(h)} &
E(Z)\ar[r]^{\alpha(Z)} \ar[d]_{E(h)} &
F(Z) \ar@{=}[r] \ar[d]^{F(h)} &
f^*F(Z) \ar[d]^{f^*F(h)} \\
f^*E(W) \ar@{=}[r] &
E(W) \ar[r]_{\alpha(W)} &
F(W) \ar@{=}[r] &
f^*F(W)
}$$
If $\beta:F\rightarrow G$ is another map of $H$-presheaves, then $f^*(\alpha\beta)=f^*\alpha f^*\beta$ by definition. 
Therefore, $f^*$ is a functor. 
Note that if $F$ is an $H$-sheaf, then $f^*F$ is an $f^*H$-sheaf. 
The restriction of $f^*$ to $H$-sheaves will also be written as $f^*$.

We complete the section with a series of lemmas concerning the properties of the functor $f^*$.
\begin{lemma}\label{pbres}
Let $H$ be a sieve on $X$ belonging to $\ZAR{X}$, $E$ an $H$-presheaf, and $Y\xrightarrow fX$ a map of schemes.
Then $(f^*E)|_Z = E|_Z$ for any object $Z\xrightarrow gY$ of $\cC_{f^*H}$.
\end{lemma}
\begin{proof}
This follows directly from the definitions of the pullback and restriction functors.
\end{proof}

\begin{lemma}\label{strpb}
Let $f:Y\rightarrow X$ and $g:Z\rightarrow Y$ be maps of schemes and $H$ a sieve on $X$ belonging to $\ZAR{X}$. Then $(fg)^*H$ and $g^*f^*H$ are the same sieves, and $(fg)^*=g^*f^*$ (equality, not natural isomorphism) as functors from $Pre_H(\ZAR{X})$ to $ Pre_{(fg)^*H}(\ZAR{Z})$.
\end{lemma}
\begin{proof}
A map $h:W\rightarrow Z$ is in the sieve $(fg)^*H$ if and only if $fgh$ is in $H$, and this condition is equivalent for $h$ to be in the sieve $g^*f^*H$. Hence $(fg)^*H=g^*f^*H$. The functors $(fg)_*$ and $f_*g_*$ from $\cC_{(fg)^*H}$ to $\cC_H$ are eqaul because both are defined by composition with $fg$. Therefore, for every $H$-presheaf $E$, $$(fg)^*E = E(fg)_*^{op}=Ef_*^{op}g_*^{op} = g^*f^*E,$$
and for every morphism $\alpha:E\rightarrow F$ and every object $W\xrightarrow h Z$ of $\cC_{(fg)^*H}$,
$$((fg)^*\alpha)(\smap Wh) = \alpha(\smap W{fgh})= (f^*\alpha)(\smap W{gh})=g^*f^*\alpha(\smap Wh).$$
\end{proof}

\begin{lemma}\label{Bf}
Suppose $f:Y\rightarrow X$ is a map of schemes and $\cF$ an $\OX$-module. Then there is a natural isomorphism $Bf^*\cF\xrightarrow \cong f^*B\cF $ where the first $B$ is the extension functor $Pre(\zar{Y})\rightarrow Pre(\ZAR{Y})$ and the second $B$ is $Pre(\zar{X})\rightarrow Pre(\ZAR{X})$.
\end{lemma}
\begin{proof}
For each scheme $Z\xrightarrow g Y$ over $Y$,
$(Bf^*\cF)(\smap Zg)=(g^*f^*\cF)(Z)$, and $(f^*B\cF)(\smap Zg) = (B\cF)(\smap Z{fg}) = ((fg)^*\cF)(Z)$.
Define $\alpha_g:(Bf^*\cF)(\smap Zg)\rightarrow (f^*B\cF)(\smap Zg)$ to be the map $(g^*f^*\cF)(Z) \rightarrow((fg)^*\cF)(Z)$ induced by the natural isomorphism $g^*f^*\rightarrow (fg)^*$. If $h:W\rightarrow Z$ is any map over $Y$, then the following diagram commutes.
$$\xymatrix{
(g^*f^*\cF)(Z) \ar[r]^{\alpha_g}_\cong \ar[d] \ar@/_2cm/[dd]_{(Bf^*\cF)(h)}&
((fg)^*\cF)(Z) \ar[d] \ar@/^2cm/[dd]^{(f^*B\cF)(h)}\\
(h^*g^*f^*\cF)(W) \ar[r]^\cong \ar[d]^\cong &
(h^*(fg)^*\cF)(W) \ar[d]_\cong \\
((gh)^*f^*\cF)(W) \ar[r]^{\alpha_{gh}}_\cong &
((fgh)^*\cF)(W)
}$$
\end{proof}

\begin{lemma}\label{xipb}
Let $f:Y\rightarrow X$ be a map of schemes, $H$ a sieve on $X$ belonging to $\ZAR{X}$.
Then $\xi_{f^*H}f^* = f^*\xi_H$ as functors.
This is a strict equality, not isomorphism.
\begin{gather*}
Pre_H(\ZAR{X})  \xrightarrow{f^*} Pre_{f^*H}(\ZAR{Y})
\xrightarrow{\xi_{f^*H}} Shv(\ZAR{Y}) \\
Pre_H(\ZAR{X})  \xrightarrow{\xi_H} Shv(\ZAR{X}) 
\xrightarrow{f^*} Shv(\ZAR{Y})
\end{gather*}
\end{lemma}
\begin{proof}
Suppose $E$ is an $H$-presheaf and $U\xrightarrow g Y$ an object of $\ZAR{Y}$.
By definition, $(\xi_{f^*H}f^*E)(U)$ is the set of equivalence classes $[\{U_i\rightarrow U\},\{s_i\}]$
such that for each $i$, $U_i\rightarrow U \xrightarrow g Y$ is an object of $f^*H$, 
$s_i\in (f^*E)^s(U_i)$, and the pullbacks of $s_i$ and $s_j$ to $U_i\times_U U_j$ coincide.
Two pairs $(\{U_i\rightarrow U\},\{s_i\})$ and $(\{V_j\rightarrow U\},\{t_j\})$ represent the same class 
if and only if the pullbacks of $s_i$ and $t_j$ to $U_i\times_U V_j$ coincide. 
Note that $U_i\rightarrow U\xrightarrow gY$ is an object of $f^*H$ 
if and only if $U_i\rightarrow U\xrightarrow {fg}X$ is an object of $H$, 
and $(f^*E)^s(U_i) = E^s(U_i)$.
So a pair $(\{U_i\rightarrow U\},\{s_i\})$ represents an element of $(\xi_{f^*H}f^*E)(U)$ 
if and only if it represents an element of $(\xi_HE)(U)=(f^*\xi_HE)(U)$. 
Also, the equivalence relations defining 
$(\xi_{f^*H}f^*E)(U)$ and $(f^*\xi_HE)(U)$ are the same. 
Therefore, $(\xi_{f^*H}f^*E)(U) = (f^*\xi_HE)(U)$. 
Next, if $h:V\rightarrow U$ is a morphisms of $\ZAR{Y}$, 
both $(\xi_{f^*H}f^*E)(h)$ and $(f^*\xi_HE)(h)$ send the class $[\{U_i\rightarrow U\},\{s_i\}]$
to the class $[\{U_i\times_UV\},\{s_i|_{U_i\times_UV}\}]$. 
Therefore, $(\xi_{f^*H}f^*E)(h)=(f^*\xi_HE)(h)$. 
This shows that $\xi_{f^*H}f^*$ and $f^*\xi_H$ agree on objects.
To show that they also agree on morphisms, 
suppose $\alpha:E\rightarrow F$ is a map of $H$-presheaves. 
For each object $U\xrightarrow gY$ of $\ZAR{Y}$, 
both $(\xi_{f^*H}f^*\alpha)(U)$ and $(f^*\xi_H\alpha)(U)$ send the class 
$[\{U_i\rightarrow U\},\{s_i\}]$ to $[\{U_i\rightarrow U\},\{\alpha^ss_i\}]$. 
Therefore, $\xi_{f^*H}f^*\alpha = f^*\xi_H\alpha$. 
This completes the proof.
\end{proof}

\section{Standard vector bundles}\label{svb}
In this section, we give the definition of the category of standard vector bundles and prove its properties.
This category is equivalent to the category of usual vector bundles and satisfies various strict functoriality.
Among other things, it has strictly functorial pullback functor and 
strictly associative tensor product, which is also strictly commutative 
with line bundles.
All sheaves are assumed to be sheaves of modules.
\subsection{The definition of standard vector bundles}\label{dsvb}
If $A$ is a commutative ring, we call an $A$-module of the form 
$$A^n=\{(a_1,\ldots,a_n)|a_i\in A, i=1,\ldots,n\}$$
a {\em standard  free} $A$-module. 
A finitely generated $A$-module is free if and only if it is isomorphic to a standard free module.
For a scheme $Y$ and a presheaf $E$ on the big Zariski site $\ZAR{Y}$,
a map $\OY[n]\rightarrow E$ is completely determined by $n$ elements of $E(Y)$, 
the images of the standard basis of $\OY(Y)^n$.
Suppose $H$ is a sieve on a scheme $X$ belonging to $\ZAR{X}$, and suppose $E$ is an $H$-presheaf.
If $Y\xrightarrow f X$ is an object of $\cC_H$, 
then $f^*H = H_Y$ where $H_Y$ is the sieve $\Hom_{\ZAR{Y}}(-,Y)$, 
and $f^*E$ is a presheaf on the big Zariski site $\ZAR{Y}$.
If, furthermore, $E(Y)$ is a standard free module, then the standard basis of $E(Y)=f^*E(Y)$ induces a map $\OY[n]\rightarrow f^*E$ of presheaves on $\ZAR{Y}$ such that the map on $Y$ is the identity.

\begin{defn}\label{defnvb}
Suppose $X$ is a scheme. A {\em standard vector bundle on $X$} is a pair $(H,E)$ where
$H$ is a sieve on $X$ that belongs to the site $\ZAR{X}$, and $E$ is an $H$-presheaf on $\ZAR{X}$ satisfying the following property: for each object $Y\xrightarrow f X$ of $\cC_H$, there exists an integer $n$ such that
$E(Y)=\OY(Y)^n$, i.e., $E(Y)$ is a standard free module, and the map $\epsilon_f:\OY[n]\rightarrow f^*E$ induced by the standard basis of $E(Y)$ is an isomorphism.
If the integer $n$ is the same for all objects of $\cC_H$, then it is called the {\em rank} of $E$. A standard vector bundle of rank 1 is called a {\em standard line bundle}.
The category of standard vector bundles on $X$ is denoted by $\bva(X)$. The set of  morphisms from $(H,E)$ to $(K,F)$ is defined to be the set of morphisms between the associated sheaves,
$$\Hom_{\bva(X)}((H,E),(K,F)) = \Hom_{Shv(\ZAR{X})}(\xi_H E, \xi_K F).$$
\end{defn}
In this definition, we required the value of $E$ at every object to be a standard free module.
This is the key requirement for the properties we wish to prove in Theorem \ref{tp}.
For simpler notation, we will sometimes write $E$ for $(H,E)$. When we do so, we will call $E$ a standard vector bundle and $H$ the associated sieve, or we will simply call $E$ an {\em $H$-vector bundle} (or {\em $H$-line bundle} if the rank is 1).
Note that $E$ is actually an $H$-sheaf, not just an $H$-presheaf since every pullback of $E$ is a sheaf. If $g:Z\rightarrow Y$ is a morphism of $\cC_H$, then $E(Y)$ and $E(Z)$ have the same rank.

\begin{example}\label{exvb}
The simplest example of standard vector bundles is the {\em trivial} standard vector bundle $\OX[n]$ of rank $n\ge0$. It is defined as an $H_X$-vector bundle.
For each object $Y\rightarrow X$ of $\ZAR{X}$, $\bOX[n](Y)=\OY(Y)^n$,
and for each map $g:Z\rightarrow Y$ over $X$, the restriction map $\bOX[n](g):\bOX[n](Y)\rightarrow \bOX[n](Z)$
is induced by the map $\OY(Y)\rightarrow \OZ(Z)$ of global sections of the structure sheaves.
The trivial standard vector bundle of rank 0 will be denoted by 0 and called the zero bundle.
\end{example}

There is a way to produce a standard vector bundle from a locally free sheaf on a scheme. The next lemma is useful for various constructions in this section. The idea is that a {\em locally free} $H$-presheaf can be standardized by choosing trivialization data.
\begin{lemma}\label{stdize}
Let $X$ be a scheme, $H$ a sieve on $X$ belonging to $\ZAR{X}$, and $E$ an $H$-presheaf.
Suppose there is an integer $n_f$ and an isomorphism $\varphi_f:\bOY[n_f]\rightarrow f^*E$ for each object $Y\xrightarrow f X$ of $\cC_H$. Then there exists an $H$-vector bundle $S_H^\varphi E$, and an isomorphism $\varphi_E:S_H^\varphi E \cong E$ induced by $\varphi$.
\end{lemma}
\begin{proof}
For an object $Y\xrightarrow f X$ of $\cC_H$, define $S_H^\varphi E(Y)=\OY(Y)^{n_f}$.
For a morphism $g:Z\rightarrow Y$ of $\cC_H$, define $S_H^\varphi E(g)$ to be the composite map
\begin{multline*}
\OY(Y)^{n_f} \xrightarrow[\cong]{\varphi_f(Y)}(f^*E)(Y)=E(Y)\\
\xrightarrow {\mathmakebox[3em]{E(g)}} E(Z)=((fg)^*E)(Z) \xrightarrow[\cong] {\varphi_{fg}^{-1}(Z)} \OZ(Z)^{n_{fg}}.
\end{multline*}
If $h$ is another morphism of $\cC_H$, then $S_H^\varphi E(gh) = S_H^\varphi E(h)S_H^\varphi E(g)$.
Hence $S_H^\varphi E$ is an $H$-presheaf.
From the way $S_H^\varphi E$ is defined on morphisms, we see that
the map $\varphi_E:S_H^\varphi E\rightarrow E$ defined by $\varphi_E(Y):S_H^\varphi E(Y)\xrightarrow{\varphi_f(Y)} E(Y)$ on each object $Y\xrightarrow f X$ of $\cC_H$ is an isomorphism of $H$-presheaves.
Let $\epsilon:\bOY[n_f] \rightarrow f^*S_H^\varphi E$ be the map induced by the standard basis of $S_H^\varphi E(Y)$. Then the following diagram commutes
$$\xymatrix{ \bOY[n_f] \ar[r]^-\epsilon \ar[dr]_{\varphi_f} & f^*S_H^\varphi E \ar[d]^{f^*\varphi_E} \\
& f^*E}$$
because the diagram of global sections commute.
$$\xymatrix{ \OY(Y)^{n_f} \ar[r]^1 \ar[dr]_{\varphi_f(Y)} &
\OY(Y)^{n_f} \ar[d]^{\varphi_f(Y)} \\
& E(Y)}$$
Since $\varphi_f$ and $f^*\varphi_E$ are isomorphisms, so is $\epsilon$.
\end{proof}

Let $X$ be a scheme and $\cE$ a locally free sheaf of finite rank on $\zar{X}$. 
We can construct a standard vector bundle from $\cE$ once we make certain choices.
Suppose that $\cU=\{U_i\rightarrow U\}$ is a covering such that 
$\cE|_{U_i}$ is a free $\cO_{U_i}$-module for each $i$. 
Let $H$ be the sieve associated to $\cU$.
If $Y\xrightarrow f X$ is an object of $\cC_H$, 
then $f$ factors as $Y\rightarrow U_i\rightarrow X$ for some $i$. 
Hence $f^*\cE$ is a free $\OY$-module of finite rank, say $n_f$.
We choose an isomorphism $\alpha_f:\OY[n_f] \rightarrow f^*\cE$ 
for every object $Y\xrightarrow f X$ of $\cC_H$. 
Since $f^*H = H_Y = f^*H_X$, $f^*(B\cE|_H) = f^*B\cE$, and by Lemma \ref{Bf}, $f^*B\cE\cong Bf^*\cE$. Then define $\varphi_f$ to be the composite map
$$\bOY[n] \cong B\OY[n] \xrightarrow[\cong]{B\alpha_f} Bf^*\cE \xrightarrow[\cong]{} f^*(B\cE|_H).$$

\begin{coro}\label{stdvb}
Let $X$ be a scheme and $\cE$ a locally free sheaf of finite rank on $\zar{X}$.
If we choose $H$ and $\varphi$ as described in the previous paragraph,
then $S_H^\varphi B\cE|_H$ is a standard vector bundle.
Moreover, there is an isomorphism $\gamma_\cE:\xi_H S_H^\varphi B\cE|_H \rightarrow B\cE$ of sheaves on $\ZAR{X}$.
\end{coro}
\begin{proof}
Applying Lemma \ref{stdize} to $E=B\cE|_H$, we get a standard vector bundle $S_H^\varphi B\cE|_H$,
and an isomorphism $\varphi_{B\cE|_H}:S_H^\varphi B\cE|_H \rightarrow B\cE|_H$.
The isomorphism $\gamma_\cE$ is the composite map
$$\xi_HS_H^\varphi B\cE|_H\xrightarrow[\cong] {\xi_H\varphi}\xi_H(B\cE|_H)\xrightarrow \cong\xi_{H_X}B\cE \xrightarrow \cong B\cE$$
where the second and the third isomorphisms are from Proposition \ref{shffres} and the fact that $B\cE$ is a sheaf (Lemma \ref{extres}(1)).
\end{proof}

Now we want to define a pullback functor $\bva(X)\rightarrow \bva(Y)$ induced by a map $f:Y\rightarrow X$ of schemes.
If $E$ is an $H$-vector bundle, then the $f^*H$-presheaf $f^*E$ is an $f^*H$-vector bundle.
To prove this, suppose $Z\xrightarrow g Y$ is an object of $f^*H$. We need to prove that $f^*E(\smap Zg)$ is a standard free module and that the map $\bOZ^n\rightarrow g^*f^*E$ induced by the standard basis of $f^*E(\smap Zg)$ is an isomorphism.
But those follow from the condition of $E$ being an $H$-vector bundle since $Z\xrightarrow g Y \xrightarrow f X$ is an object of $H$, $f^*E(\smap Zg) = E(\smap Z{fg})$, and $g^*f^*E=(fg)^*E$ by Lemma \ref{strpb}. Therefore we can define the pullback of $(H,E)$ to be $(f^*H,f^*E)$.
Suppose $\alpha:(H,E)\rightarrow (K,F)$ is a morphism of standard vector bundles in $\bva(X)$, that is, a morphism $\alpha:\xi_HE\rightarrow \xi_KF$ of sheaves.
Then $f^*\alpha$ is a morphism $f^*\xi_HE\rightarrow f^*\xi_KF$, which is a morphism
$\xi_{f^*H}f^*E \rightarrow \xi_{f^*K}f^*F$ by Lemma \ref{xipb}. So it is a map $(f^*H,f^*E)\rightarrow (f^*K,f^*F)$ of standard vector bundles. The pullback of the map $\alpha$ of standard vector bundles is defined to be $f^*\alpha$. If $\beta:(K,F)\rightarrow (L,G)$ is another map of standard vector bundles, then $f^*(\beta\alpha) = f^*\beta f^*\alpha$ since $f^*$ is a functor. Therefore, we have defined a functor $\bva(X)\rightarrow \bva(Y)$. We will denote it by $f^*$ as well.
\begin{prop}\label{functvb}
Suppose $f:Y\rightarrow X$ and $g:Z\rightarrow Y$ are maps of schemes. Then $(fg)^*=f^*g^*$ as functors $\bva(Z)\rightarrow \bva(X)$. (This is a strict equality, not a natural isomorphism.)
\end{prop}
\begin{proof}
This follows from the definition of the pullback functors and Lemma \ref{strpb}.
\end{proof}

\begin{lemma}\label{Bsmall}
Suppose $(H,E)$ is a standard vector bundle on $X$. If $f:Y\rightarrow X$ is a map of schemes,
then the induced map $\lambda:f^*(\xi_HE|_X) \rightarrow \xi_HE|_Y$ is an isomorphism.
\end{lemma}
\begin{proof}
We prove this by showing that the map at every stalk is an isomorphism. Suppose $y\in Y$ and $x = f(y)$.
We can choose an open subscheme $U\subset X$ containing $x$ such that the inclusion $i:U\rightarrow X$ is in the sieve $H$.
Let $V$ be an open subscheme of $Y$ containing $f^{-1}(U)$, and $j:V\rightarrow Y$ the inclusion, and $g=f|_V :V\rightarrow U$.
$$\xymatrix{
V \ar[r]^j \ar[d]_g &Y\ar[d]^f \\
U \ar[r]_i &X }$$
Since $E$ is an $H$-vector bundle, there is an isomorphism $\epsilon:\bOU[n] \rightarrow i^*E$ induced by the standard basis of $E(U)$. Since $i^*E|_U = E|_U$ and $i^*E|_V = E|_V$ by Lemma \ref{pbres}, 
we obtain the following commutative diagram, which shows that the induced map $\lambda_E:g^*E|_U\rightarrow E|_V$ is an isomorphism.
$$\xymatrix{
\OV[n] \ar[r]^-\cong \ar[dr]_1&
g^*\OU[n] \ar[r]^{g^*\epsilon|_U}_\cong \ar[d]^{\lambda_{\bO}}&
g^*E|_U \ar[d]^{\lambda_E}\\
& \OV[n] \ar[r]_{\epsilon|_V}^\cong &
E|_V
}$$
Now $j^*f^*(\xi_HE|_X) \cong g^*i^*(\xi_HE|_X) \cong g^*(\xi_HE|_U) \cong g^*E|_U$ and $j^*(\xi_HE|_Y) \cong \xi_HE|_V\cong E|_V$. Thus we have the following commutative diagram, which shows that $\lambda|_V$ is an isomorphism.
$$\xymatrix{f^*(\xi_HE|_X)|_V \ar[r]^-{\lambda|_V} \ar[d]_\cong&
\xi_HE|_V \ar[d]^\cong \\
g^*E|_U \ar[r]_{\lambda_E}^\cong & E|_V
}$$
Therefore, the localized map $\lambda_y$ is an isomorphism as it is the localization of the top row at $y$.
\end{proof}

\subsection{Direct sum and tensor product}
We will define two bifunctors $\oplus: \bva(X)\times\bva(X)\rightarrow \bva(X)$ and $\otimes:\bva(X)\times\bva(X)\rightarrow \bva(X)$ called direct sum and tensor product of standard vector bundles.
First we define presheaf versions.
$$\begin{array}{c}
\noplus:Pre_H(\ZAR{X})\times Pre_H(\ZAR{X})\rightarrow Pre_H(\ZAR{X})\\
\notimes:Pre_H(\ZAR{X})\times Pre_H(\ZAR{X})\rightarrow Pre_H(\ZAR{X})
\end{array}$$
If $E$ and $F$ are $H$-presheaves, then define $(E\noplus F)(Y)=E(Y)\oplus F(Y)$ for each object $Y$, and $(E\noplus F)(g) = E(g)\oplus F(g)$ for each morphism $g$.
If $\gamma:E\rightarrow E'$ and $\delta:F\rightarrow F'$ are maps of presheaves, then
$$(\gamma\noplus \delta)(Y)=\gamma(Y)\oplus \delta(Y):E(Y)\oplus F(Y)\rightarrow E'(Y)\oplus F'(Y).$$
Similarly, define $(E\notimes F)(Y) = E(Y)\otimes_{\OY(Y)}F(Y)$ on objects,  $(E\notimes F)(g) = E(g)\otimes F(g)$ on morphisms, and
$$(\gamma\notimes\delta)(Y) = \gamma(Y)\otimes \delta(Y):E(Y)\otimes_{\OY(Y)} F(Y)\rightarrow E'(Y)\otimes_{\OY(Y)} F'(Y).$$

Let $H$ and $K$ be sieves on $X$ that belong to $\ZAR{X}$, and let $E$ be an $H$-vector bundle and $F$ a $K$-vector bundle.
We will define their direct sum $E\oplus F$ as an $H\cap K$-vector bundle.
The presheaf direct sum $E\noplus F$ is not a standard vector bundle since the value at an object is not a standard free module.
But we can make it into one through a standardization process (Lemma \ref{stdize}).
For each object $Y\xrightarrow f X$ of $\cC_{H\cap K}$,
we have isomorphisms $\alpha:\bOY[r]\rightarrow f^*E$ and $\beta:\bOY[s]\rightarrow f^*F$
induced by the standard bases of $E(Y)$ and $F(Y)$.
Let $\varphi_f$ be the composite map
$$\varphi_f: \bOY[r+s]\xrightarrow[\cong]\sigma \bOY[r]\noplus \bOY[s] \xrightarrow[\cong]{\alpha\noplus \beta} f^*E\noplus f^*F=f^*(E\noplus F)$$
where $\sigma$ is the isomorphism
\begin{equation}\label{sigma}
(a_1,\ldots,a_r,a_{r+1},\ldots,a_{r+s})\mapsto((a_1,\ldots,a_r),(a_{r+1},\ldots,a_{r+s})).
\end{equation}
Then define $E\oplus F = S_{H\cap K}^\varphi (E\noplus F)$.
Since $E\oplus F\cong E\noplus F$, there is an isomorphism
\begin{equation*}
\omega:\xi_{H\cap K}(E\oplus F)\cong \xi_{H\cap K}(E\noplus F)\cong \xi_H E\noplus \xi_K F.
\end{equation*}
If $\gamma:(H,E)\rightarrow (H'E')$ and $\delta:(K,F)\rightarrow (K',F')$ are maps of standard vector bundles, that is,
maps $\gamma:\xi_HE\rightarrow \xi_{H'}E'$ and $\delta:\xi_KF\rightarrow \xi_{K'}F'$ of associated sheaves, then $\gamma\oplus \delta$ is defined to be the following composite map.
$$\xi_{H\cap K}(E\oplus F) \xrightarrow\omega \xi_HE \noplus \xi_K F\xrightarrow {\gamma\noplus \delta} \xi_{H'}E'\noplus \xi_{K'}F' \xrightarrow {\omega^{-1}} \xi_{H'\cap K'}(E'\oplus F')$$
If $\gamma':(H',E')\rightarrow (H',E'')$ and $\delta':(K',F')\rightarrow (K'',F'')$ are another pair of maps of standard vector bundles, then $(\gamma'\oplus\delta')(\gamma\oplus\delta) = \gamma'\gamma \oplus \delta'\delta$ since a similar formula for $\noplus$ holds.
Therefore, we have defined a bifunctor $\oplus:\bva(X)\times\bva(X)\rightarrow \bva(X)$.
The isomorphism $\omega$ also allows us to define projections and injections of standard vector bundles
\begin{align*}
p_E&:E\oplus F \rightarrow E & \quad i_E&:E\rightarrow E\oplus F \\
p_F&:E\oplus F \rightarrow F & \quad i_F&:F\rightarrow E\oplus F
\end{align*}
such that $i_Ep_E+i_Fp_F=1_{E\oplus F}$, $p_Ei_E=1_E$, $p_Fi_F=1_F$, $p_Ei_F=0$, and $p_Fi_E=0$.
So the direct sum operation $\oplus$ is a biproduct operation in $\bva(X)$.
This construction can be generalized to the direct sum of multiple terms.
The category $\bva(X)$ is an additive category with $\oplus$ as the biproduct operation.

The construction of the tensor product $\otimes:\bva(X)\times\bva(X)\rightarrow \bva(X)$ is similar.
If $E$ is an $H$-vector bundle and $F$ is a $K$-vector bundle, then $E\otimes F$ will be an $H\cap K$-vector bundle.
For each object $Y\xrightarrow f X$ of $\cC_{H\cap K}$, let $\alpha:\OY[r] \rightarrow f^*E$ and $\beta:\OY[s]\rightarrow f^*F$ be the isomorphisms induced by the standard bases of $E(Y)$ and $F(Y)$.
Define $\varphi_f$ to be the composite map
$$\varphi_f : \bOY[rs] \xrightarrow[\cong]{\pi^{-1}} \bOY[r]\notimes\bOY[s] \xrightarrow[\cong]{\alpha\notimes\beta}
f^*E\notimes f^*F = f^*(E\notimes F)$$
where $\pi$ is the isomorphism
\begin{equation}\label{pi}
(a_1,\ldots,a_r)\otimes(b_1,\ldots,b_s)\mapsto (a_1b_1,\ldots,a_1b_s,\ldots,a_rb_1,\ldots,a_rb_s).
\end{equation}
Using Lemma \ref{stdize} with this collection of isomorphisms, define the tensor product of $E$ and $F$ to be $E\otimes F = S_{H\cap K}^\varphi(E\notimes F)$. Suppose $\gamma:(H,E)\rightarrow (H'E')$ and $\delta:(K,F)\rightarrow (K',F')$ are maps of standard vector bundles. They are the maps $\gamma:\xi_HE\rightarrow \xi_{H'}E'$ and $\delta:\xi_KF\rightarrow \xi_{K'}F'$ of the associated sheaves.
Since $E$ is an $H$-sheaf, there is a natural isomorphism $E\cong (\xi_HE)|_H$ by Lemma \ref{unitmap}. There are similar isomorphisms for other standard vector bundles as well. Then $E\notimes F \cong (\xi_HE)|_H\notimes(\xi_KF)|_K =\xi_HE|_{H\cap K}\notimes \xi_HF|_{H\cap K} = (\xi_HE\notimes\xi_KF)|_{H\cap K}$, and by Proposition \ref{shffres}, there is an isomorphism $\zeta$ defined by composing a series of isomorphisms.
\begin{multline}\label{shftp}
\zeta:\xi_{H\cap K}(E\otimes F)\cong\xi_{H\cap K}(E\notimes F)\\
\cong \xi_{H\cap K}((\xi_HE\notimes\xi_KF)|_{H\cap K}) \cong \xi_{H_X}(\xi_HE\notimes\xi_KF).
\end{multline}
Now the map $\gamma\otimes\delta$ is defined to be the composite map
\begin{multline*}
\gamma\otimes\delta: \xi_{H\cap K}(E\otimes F) \xrightarrow{\mathmakebox[3em]{\zeta}} \xi_{H_X}(\xi_HE\notimes \xi_KF) \\
\xrightarrow{\mathmakebox[3em]{\xi(\gamma\notimes\delta)}} \xi_{H_X}(\xi_{H'}E'\notimes \xi_{K'}F')
\xrightarrow{\mathmakebox[3em]{\zeta^{-1}}}\xi_{H'\cap K'}(E'\otimes F')
\end{multline*}
If $\gamma':(H',E')\rightarrow (H',E'')$ and $\delta':(K',F')\rightarrow (K'',F'')$ are another pair of maps of standard vector bundles, then $(\gamma'\otimes\delta')(\gamma\otimes\delta) = \gamma'\gamma \otimes \delta'\delta$ since a similar formula for $\notimes$ holds.
Thus, we have defined a bifunctor $\otimes:\bva(X)\times\bva(X)\rightarrow \bva(X)$.

\begin{thm}\label{ds}
Let $X$ be a scheme and $\bva(X)$ the category of standard vector bundles on $X$.
\begin{enumerate}
\item The direct sum $\oplus:\bva(X)\times\bva(X)\rightarrow \bva(X)$ is strictly associative. In other words, the following diagram commutes (strictly, not up to a natural isomorphism).
$$\xymatrix{\bva(X)\times \bva(X)\times\bva(X) \ar[rr]^-{\oplus\times 1} \ar[d]_{1\times\oplus}&&
\bva(X)\times\bva(X) \ar[d]^\oplus \\
\bva(X)\times\bva(X) \ar[rr]_-\oplus && \bva(X)}$$
\item The zero bundle 0 is the strict identity with respect to $\oplus$, In other words,
 for any $E\in\bva(X)$, $0\oplus E = E\oplus 0 = E$ (identities, not natural isomorphisms), and if $\gamma:E\rightarrow F$ is a map of standard vector bundles, then $1_0\oplus \gamma = \gamma\oplus 1_0 = \gamma$.
\item If $f:Y\rightarrow X$ is a map of schemes, then $f^*$ preserves $\oplus$ and the identity object. In other words,
$f^*0= 0$, and the following diagram commutes
$$\xymatrix{
\bva(X)\times\bva(X) \ar[r]^-\oplus \ar[d]_{(f^*,f^*)} &
\bva(X) \ar[d]^{f^*} \\
\bva(Y)\times\bva(Y) \ar[r]_-\oplus & \bva(Y)
}$$
\end{enumerate}
\end{thm}
\begin{proof}
These statements are about the commutativity of various diagrams of functors.
Two composite functors are the same when they agree on objects and on morphisms.
First, the equality of objects, i.e., standard vector bundles, (which are functors,)  is shown by proving that they have equal modules of sections and equal restriction maps.
Since the modules of sections of standard vector bundles are standard free modules, two of them are the same if and only if they have the same rank. It can be verified easily.
So we only need to see if they have the same restriction maps.
Suppose $(H,E)$ is a standard vector bundle, and $g:Z\rightarrow Y$ is a morphism of $\cC_H$.
Since $E(Y)$ and $E(Z)$ are standard free modules,
the map $E(g):E(Y)\rightarrow E(Z)$ is represented by a matrix (with respect to the standard bases).
Suppose $K$ is the sieve associated to $F$, and $g$ is in $\cC_{H\cap K}$.
Then $(E\oplus F)(g)$ is represented by the block matrix
$$\quadmat {E(g)}00{F(g)}$$
because the standard basis of $(E\oplus F)(W)$ corresponds to those of $E(W)$ and $F(W)$ via the isomorphism
$$\sigma:(a_1,\ldots,a_r,a_{r+1},\ldots,a_{r+s})\mapsto ((a_1,\ldots,a_r),(a_{r+1},\ldots,a_{r+s}))$$
for all relevant objects $W$. 
Thus the commutativity of the diagrams on objects follows.

For the commutativity of the first diagram on morphisms, suppose $\gamma:(H,E)\rightarrow (H'E')$, $\delta:(K,F)\rightarrow (K',F')$, and $\varepsilon:(L,G)\rightarrow (L',G')$ are morphisms of standard vector bundles. We need to show $(\gamma\oplus\delta)\oplus\varepsilon = \gamma\oplus(\delta\oplus\varepsilon)$. It suffices to show the commutativity of the following diagram as then the back square shows the equality.
$$\xymatrix@C=-5pt{
\xi((E\oplus F)\oplus G) \ar[rr]^{(\gamma\oplus\delta)\oplus\varepsilon} \ar[dr]_\cong \ar[ddd]_1 
&&
\xi((E'\oplus F')\oplus G') \ar[dr]^\cong \ar'[d][ddd]_1\\
& (\xi E\noplus \xi F)\noplus \xi G \ar[rr]^(.5){(\gamma\noplus\delta)\noplus\varepsilon} \ar[ddd]_(.4)\alpha 
&&
(\xi E'\noplus \xi F')\noplus \xi G' \ar[ddd]^\alpha \\\\
\xi(E\oplus (F\oplus G)) \ar'[r][rr]^(.25){\gamma\oplus(\delta\oplus\varepsilon)} \ar[dr]_\cong 
&&
\xi(E'\oplus (F'\oplus G')) \ar[dr]^\cong \\
& \xi E\noplus (\xi F \noplus \xi G) \ar[rr]_{\gamma\noplus(\delta\noplus\varepsilon)}
&&
\xi E'\noplus(\xi F'\noplus \xi G')
}$$
In the diagram, the isomorphism $\alpha$ is the associativity isomorphism
\begin{align*}
(((a_1,\ldots,a_r),&\,(b_1,\ldots,b_s)),(c_1,\ldots,c_t))\\
&\longmapsto ((a_1,\ldots,a_r),((b_1,\ldots,b_s),(c_1,\ldots,c_t)))
\end{align*}
so the front square commutes. The slanted arrows are derived from the isomorphism $\sigma$ defined by (\ref{sigma}). Therefore, the left and right squares commute. The top and bottom squares commute by definition. Therefore, the whole diagram commutes.
The property (2) of the theorem is proved similarly. The property (3) follows directly from the definition of $f^*$ since $f^*E$ and $f^*\gamma$ are the same as $E$ and $\gamma$ everywhere they are defined for any standard vector bundle $E$ and any map $\gamma$ of standard vector bundles.
\end{proof}

\begin{thm}\label{tp}
Let $X$ be a scheme and $\bva(X)$ the category of standard vector bundles on $X$.
\begin{enumerate}
\item \label{satp} The tensor product $\otimes:\bva(X)\times\bva(X)\rightarrow \bva(X)$ is strictly associative. In other words, the following diagram commutes (strictly, not up to a natural isomorphism).
$$\xymatrix{\bva(X)\times \bva(X)\times\bva(X) \ar[rr]^-{\otimes\times 1} \ar[d]_{1\times\otimes}&&
\bva(X)\times\bva(X) \ar[d]^\otimes \\
\bva(X)\times\bva(X) \ar[rr]_-\otimes && \bva(X)}$$
\item The trivial standard line bundle $\bOX$ is the strict identity with respect to $\otimes$, In other words,
 for any $E\in\bva(X)$, $\bOX\otimes E = E\otimes \bOX = E$ (identities, not natural isomorphisms), and if $\gamma:E\rightarrow F$ is a map of standard vector bundles, then $1_{\bOX}\otimes \gamma = \gamma\otimes 1_{\bOX} = \gamma$.
\item \label{commlb} Let $\lbdl(X)$ be the category of standard line bundles on $X$, a full subcategory of $\bva(X)$. Then $\lbdl(X)$ is a strict center in the sense that
$E\otimes L = L \otimes E$ (identity, not natural isomorphism) for all $E\in\bva(X)$ and $L\in\lbdl(X)$, and $\gamma\otimes\beta=\beta\otimes\gamma$ for all morphisms $\gamma$ of $\bva(X)$ and $\beta$ of $\lbdl(X).$
\item \label{ftp} If $f:Y\rightarrow X$ is a map of schemes, then $f^*$ preserves $\otimes$ and the identity object. In other words,
$f^*\bOX = \bOY$, and the following diagram commutes
$$\xymatrix{
\bva(X)\times\bva(X) \ar[r]^-\otimes \ar[d]_{(f^*,f^*)} &
\bva(X) \ar[d]^{f^*} \\
\bva(Y)\times\bva(Y) \ar[r]_-\otimes & \bva(Y)
}$$
\end{enumerate}
\end{thm}
\begin{proof}
This theorem is analogous to the previous theorem on direct sums. So the idea of the proof is the same.
It is worth to note that if $(H,E)$ and $(K,F)$ are standard vector bundles, and $g$ is a morphism in $\cC_{H\cap K}$, then $(E\otimes F)(g)$ is represented by the tensor product of the matrices representing $E(g)$ and $F(g)$ where the tensor product of two matrices $A$ and $B$ is defined to be the following block matrix.
$$\left(\begin{matrix}
a_{11}B & a_{12}B & \cdots \\
a_{21}B & a_{22}B & \cdots \\
\vdots & \vdots& \ddots
\end{matrix}\right)$$
Therefore, object-wise, $0\otimes E=E\otimes 0=0$ since all involved matrices are empty matrices, (1) is true since $(A\otimes B)\otimes C = A\otimes(B\otimes C)$ for any matrices $A$,$B$, and $C$.
(2) is true since $\bOX(g)$ is the 1-by-1 matrix 1, (3) is true since for any standard line bundle $L$, $L(g)$ is a 1-by-1 matrix, and (4) is true since $f^*E$ is the same as $E$ everywhere it is defined.

To prove the equation of (2) on morphisms, suppose $\gamma:(H,E)\rightarrow (K,F)$ is a morphism, then there is a commutative diagram
$$\xymatrix@C=5pt{
\xi_H(\bOX\otimes E) \ar[rr]^{1\otimes\gamma} \ar[dd]_1 \ar[dr]_(.4)\zeta &&
\xi_{K}(\bOX\otimes F) \ar[dr]^\zeta \ar'[d]_(.7)1[dd] \\
&\xi_{H_X}(\bOX\notimes \xi_H E) \ar[rr]_(.4){\xi(1\notimes\gamma)} \ar[dl]_(.6)\mu&&
\xi_{H_X}(\bOX\notimes\xi_K F) \ar[dl]^\mu \\
\xi_H E \ar[rr]_{\gamma} &&
\xi_K F
}$$
In this diagram, $\zeta$ is the isomorphism (\ref{shftp}), which was derived from the isomorphism $\pi^{-1}$ where $\pi$ is the isomorphism defined by (\ref{pi}), and $\mu$ is the isomorphism derived by $\pi$. Therefore, the triangles commute. The top and the bottom squares commute by definition. Therefore, the back square commutes, and $1\otimes\gamma=\gamma$.
The commutativity of the diagram in (1) is proved similarly. If $(\gamma,\delta,\varepsilon)$ is a morphism of $\bva(X)\times\bva(X)\times\bva(X)$, then the following diagram similar to the diagram used in the proof of the previous theorem commutes.
$$\xymatrix@C=-10pt{
\xi((E\otimes F)\otimes G) \ar[rrr]^{(\gamma\otimes\delta)\otimes\varepsilon} \ar[dr]_\cong \ar[ddd]_1 
&&&
\xi((E'\otimes F')\otimes G') \ar[dr]^\cong \ar'[d][ddd]_1\\
& \xi(\xi(\xi E\notimes \xi F)\notimes \xi G) \ar[rrr]^(.47){\xi(\xi(\gamma\notimes\delta)\notimes\varepsilon)} \ar[ddd]_(.4)\alpha 
&&&
\xi(\xi(\xi E'\notimes \xi F')\notimes \xi G') \ar[ddd]^\alpha \\\\
\xi(E\otimes (F\otimes G)) \ar'[r][rrr]^(.22){\gamma\otimes(\delta\otimes\varepsilon)} \ar[dr]_\cong 
&&&
\xi(E'\otimes (F'\otimes G')) \ar[dr]^\cong \\
& \xi(\xi E\notimes \xi(\xi F \notimes \xi G)) \ar[rrr]_{\xi(\gamma\notimes\xi(\delta\notimes\varepsilon))}
&&&
\xi(\xi E'\notimes\xi(\xi F'\notimes \xi G'))
}$$
Note that for the commutativity of the left and the right squares, we use the fact that
$\pi(\pi(u,v),w) = \pi(u,\pi(v,w))$ for any three vectors $u,v,$ and $w$.
The property (3) follows from the commutativity of the next diagram.
$$\xymatrix{
\xi(E\otimes L) \ar[rr]^{\gamma\otimes\beta} \ar[dd]_1 \ar[dr]_\cong &&
\xi(E'\otimes L') \ar'[d][dd]_(.3)1 \ar[dr]^\cong \\
& \xi(\xi E\notimes\xi L) \ar[rr]^(.35){\xi(\gamma\notimes\beta)} \ar[dd]_(.3)\tau &&
\xi(\xi E'\notimes\xi L') \ar[dd]^\tau \\
\xi(L\otimes E) \ar'[r][rr]_(.25){\beta\otimes\gamma} \ar[dr]_\cong &&
\xi(L'\otimes E') \ar[dr]^\cong \\
& \xi(\xi L\notimes\xi E)\ar[rr]_{\xi(\beta\notimes\gamma)} &&
\xi(\xi L'\otimes\xi E')
}$$
For the commutativity of the left and the right squares,
we need the fact that $\pi(u,v)=\pi(v,u)$ if $u$ or $v$ is a 1-dimensional vector.
The property (5) follows from the fact that $f^*E$ and $f^*\gamma$ are the same as $E$ and $\gamma$ everywhere they are defined for any standard vector bundle $E$ and any map $\gamma$ of standard vector bundles.
\end{proof}

\begin{rmk}\label{notcommgen}
The reason the tensor product is not strictly commutative in general is that
for a commutative ring $A$, a choice needs to be made 
to define an isomorphism $A^r\otimes A^s \rightarrow A^{rs}$,
and no choice is symmetric unless $r\le1$ or $s\le1$.
For example, if $a, b, c$, and $d$ are elements of $A$, then
$(a, b)\otimes (c, d)=\left( \begin{array}{cc}
ac & ad\\
bc & bd
        \end{array}
\right)
$
and
$(c, d)\otimes (a, b)=\left( \begin{array}{cc}
ac & bc\\
ad & bd
\end{array}
\right)
$.
Thus, for an arbitrary isomorphism $f:A^r\otimes A^s \rightarrow A^{rs}$,
we cannot expect $f((a,b)\otimes(c,d))$ and $f((c,d)\otimes(a,b))$ to be
equal since $ad\ne bc$ in general.
\end{rmk}

\begin{thm}\label{thmvb}
Let $X$ be a scheme and $\bva(X)$ the category of standard vector bundles on $X$.
\begin{enumerate}
\item $\bva(X)$ is a small exact category.
\item Let $\vbundle{X}$ be the category of locally free $\OX$-modules of finite rank.
There are exact functors
$\Phi:\bva(X)\rightarrow \vbundle{X}$ and $\Psi:\vbundle{X}\rightarrow \bva(X)$
that are equivalences of categories.
\item \label{functvbitem} If $f:Y\rightarrow X$ is a map of schemes, then $f^*:\bva(X)\rightarrow \bva(Y)$ is an exact functor. If $g:Z\rightarrow Y$ is another map of schemes, then $g^*f^*=(fg)^*$ as functors
$\bva(X)\rightarrow \bva(Z)$. (It is an equality, not simply a natural isomorphism.)
\item \label{tpzero} The tensor product $\otimes:\bva(X)\times\bva(X)\rightarrow \bva(X)$ is a biexact pairing, in other words,
for any $E\in\bva(X)$, $0\otimes E=E\otimes0=0$, and if $\cS$ is a short exact sequence of $\bva(X)$, then so are $\cS\otimes E$ and $E\otimes \cS$.
\end{enumerate}
\end{thm}
\begin{proof}
The category $\bva(X)$ is small because $\ZAR{X}$ is small and the values of a standard vector bundle at objects are standard free modules. It will be shown to be an exact category later.

Define a functor $\Phi:\bva(X)\rightarrow \vbundle{X}$ as follows.
Suppose $E$ is a standard vector bundle on $X$ with the associated sieve $H$.
Define $\Phi E=\xi_H E|_X$, the restriction of the sheafification of $E$ to the small Zariski site of $X$.
There is a Zariski covering $\{U_i\xrightarrow{f_i} X\}$ with $f_i$ in $\cC_H$ since $H$ belongs to $\ZAR{X}$. For each $i$, we have an isomorphism $\bO_{U_i}^n\rightarrow f_i^*E$. Restricting it to the small Zariski site of $U_i$, we get an isomorphism $\cO_{U_i}^n\rightarrow (f_i^*E)|_{U_i} = E|_{U_i}$.
Applying the sheafification functor, we get an isomorphism $\cO_{U_i}^n\rightarrow \xi(E|_{U_i})$. By Proposition \ref{shffresii}, $\Phi E|_{U_i}=(\xi_H E|_X)|_{U_i} = \xi(E|_{U_i})$ Hence, $\Phi E$ is indeed a locally free sheaf.
If $(K,F)$ is another standard vector bundle on $X$ and $\alpha:(H,E)\rightarrow (K,F)$ is a morphism, that is, a map $\alpha:\xi_H E\rightarrow \xi_K F$ of sheaves on $\ZAR{X}$, then $\Phi(\alpha)$ is defined to be the induced map
$\alpha|_X:\xi_H E|_X \rightarrow \xi_K F|_X$ of sheaves on $X$. This assignment respects the composition of morphisms since the restriction $-|_X$ is a functor. Hence $\Phi$ is a functor.

Next, we define the inverse $\Psi:\vbundle{X}\rightarrow \bva(X)$.
If $\cE$ is a locally free sheaf, then define $\Psi\cE=S_H^\varphi B\cE|_H$.
(See Lemma \ref{stdize} and Corollary \ref{stdvb}.)
Note that a choice of a sieve $H$ and a collection of isomorphisms $\varphi$ needs to be made for each $\cE$.
Suppose that $\cF$ is another locally free sheaf and that $\beta:\cE \rightarrow \cF$ is a map of sheaves.
If the sieve and the isomorphisms associated to $\cF$ are $K$ and $\psi$, then $\Psi \cF = S_K^\psi B\cE|_K$,
and there are isomorphisms of sheaves $\gamma_\cE:\xi_H\Psi\cE \rightarrow B\cE$ and $\gamma_\cF:\xi_K\Psi \cF \rightarrow B\cF$ by Corollary \ref{stdvb}.
Using them, define $\Psi\beta=\gamma_\cF^{-1}(B\beta)\gamma_\cE$.
$$\Psi\beta:\xi_H\Psi \cE\xrightarrow {\gamma_\cE}B\cE \xrightarrow{B\beta} B\cF \xrightarrow {\gamma_\cF^{-1}}\xi_K\Psi \cF$$
It was defined in such a way that the following diagram commutes, so that we may identify the sheafification of $\Psi\cE$ with $B\cE$ intrinsically without reference to the choice of $H$ and $\varphi$.
\begin{equation}\label{psimor}
\xymatrix{
\xi_H\Psi\cE \ar[r]^{\gamma_\cE} \ar[d]_{\Psi\alpha} &
B\cE \ar[d]^{B\alpha} \\
\xi_K\Psi\cF \ar[r]_{\gamma_\cF} &
B\cF
}
\end{equation}
This assignment respects the composition of morphisms since $B$ is a functor.
Therefore, $\Psi$ is a functor.

Now we prove that $\Phi$ and $\Psi$ are inverses to each other.
Suppose $\cE\in \vbundle{X}$. By Lemma \ref{extres}, there is an isomorphism
$$\Phi\Psi\cE=(\xi_H\Psi\cE)|_X\xrightarrow[\cong]{\gamma_\cE|_X} B\cE|_X \cong 1_X^*\cE \cong \cE.$$
This isomorphism is natural in $\cE$ since for a morphism $\alpha:\cE\rightarrow \cF$ of $\vbundle{X}$,
the following diagram commutes. (The left square commutes by definition, the middle by (\ref{psimor}), and the right by the naturality of the isomorphism of Lemma \ref{extres}.)
$$\xymatrix{
\Phi\Psi\cE \ar@{=}[r] \ar[d]_{\Phi\Psi\alpha} &
\xi_H\Psi\cE|_X \ar[r]^{\gamma_\cE|_X} \ar[d]_{(\Psi\alpha)|_X} &
B\cE|_X \ar[r]^\cong \ar[d]^{(B\alpha)|_X} & \cE \ar[d]^{\alpha} \\
\Phi\Psi\cF \ar@{=}[r] &
\xi_H\Psi\cF|_X \ar[r]_{\gamma_\cF|_X} &
B\cF|_X \ar[r]_\cong & \cF \\
}$$
Therefore, $\Phi\Psi$ is naturally isomorphic to the identity functor on $\vbundle{X}$.
Conversely, suppose $E$ is an $H$-vector bundle, and suppose $\Psi\Phi E$ turns out to be a $K$-vector bundle.
The isomorphism
$$\gamma_{\Phi \cE}:\xi_K\Psi\Phi E \rightarrow B\Phi E = B(\xi_HE|_X)$$
of Corollary \ref{stdvb} is natural in $E$ by the diagram (\ref{psimor}).
In addition, there is a natural isomorphism $B(\xi_HE|_X)\cong \xi_H E$ by Lemma \ref{resext} and Lemma \ref{Bsmall}.
Composing them together, we obtain a natural isomorphism $\xi_K\Psi\Phi E\cong\xi_H E$. Therefore, $\Psi\Phi$ is naturally isomorphic to the identity functor. This proves the equivalence of $\vbundle{X}$ and $\bva(X)$.

The category $\bva(X)$ of standard vector bundles is additive with $\oplus$ as the biproduct operation.
The functors $\Phi$ and $\Psi$ are additive since
$$\begin{array}{c}
\Phi(E\oplus F) = \xi_{H\cap K} (E\oplus F)|_X \cong \xi_H E|_X\oplus \xi_K F|_X=\Phi E\oplus \Phi F,\\
\Psi(\cE\oplus \cF) \cong B(\cE\oplus \cF) \cong B\cE \oplus B\cF \cong \Psi\cE \oplus \Psi\cF,
\end{array}$$
and projections and injections are preserved.
The category $\vbundle{X}$ is well known to be an exact category, (as a full subcategory of the abelian category of $\OX$-modules closed under extensions,)
and  $\bva(X)$ is equivalent to $\vbundle{X}$.
Therefore, $\bva(X)$ can be given the structure of an exact category such that the equivalences $\Phi$ and $\Psi$ become exact functors by transporting the notion of exactness from $\vbundle{X}$ to $\bva(X)$, that is, a sequence $0\rightarrow E\rightarrow F\rightarrow G \rightarrow 0$ of standard vector bundles is defined to be exact if and only if  $0\rightarrow \Phi E\rightarrow \Phi F \rightarrow \Phi G \rightarrow 0$ is exact.

To prove the third property of the theorem, suppose $f:Y\rightarrow X$ is a map of schemes. For any standard vector bundle $(H,E)$ in $\bva(X)$, we have a natural isomorphism $$\Phi f^*E = \xi_{f^*H} f^*E|_Y = f^*\xi_HE|_Y=\xi_H E|_Y \xleftarrow\cong f^*(\xi_HE|_X)=f^*\Phi E$$
by the definition of $\Phi$, Lemma \ref{xipb}, and Lemma \ref{pbres}.
If $0\rightarrow E\rightarrow F\rightarrow G\rightarrow 0$ is a short exact sequence in $\bva(X)$, then the sequence $0\rightarrow \Phi E\rightarrow \Phi F \rightarrow \Phi G\rightarrow 0$ is exact. Hence $0\rightarrow f^*\Phi E\rightarrow f^*\Phi F\rightarrow f^*\Phi G\rightarrow 0$ is exact, and so is $0\rightarrow \Phi f^* E\rightarrow \Phi f^*F\rightarrow \Phi f^*G\rightarrow 0$. Therefore,
$0\rightarrow f^*E\rightarrow f^*F\rightarrow f^*G\rightarrow 0$ is an exact sequence in $\bva(Y)$. This proves that $f^*:\bva(X)\rightarrow \bva(Y)$ is an exact functor. If $g:Z\rightarrow Y$ is another map of schemes, then $(fg)^*=g^*f^*$ by Proposition \ref{functvb}.

Finally, we prove that $\otimes$ is biexact.
First note that for any any scheme $Y$ over $X$ and any standard vector bundles $(H,E)$ and $(K,F)$ on $X$,
there is an isomorphisms.
\begin{align}
\label{shftpsmi} \xi_{H\cap K}(E\otimes F)|_Y &\cong \xi_{H_X}(\xi_HE\notimes \xi_KF)|_Y \\
\label{shftpsmii} &= \xi((\xi_HE\notimes \xi_KF)|_Y) \\
\label{shftpsmiii}&= \xi(\xi_HE|_Y \notimes \xi_KF|_Y) \\
\label{shftpsmiv} &= \xi_HE|_Y\otimes_{\OY} \xi_KF|_Y
\end{align}
We used the isomorphism (\ref{shftp}) for (\ref{shftpsmi}), Proposition \ref{shffresii} for (\ref{shftpsmii}) and the definition of the tensor product of $\OY$-modules for (\ref{shftpsmiv}). If $\alpha:(H,E)\rightarrow (H',E')$ and $\beta:(K,F)\rightarrow (K',F')$ are maps of standard vector bundles, then the map $(\alpha\otimes\beta)|_Y$ corresponds to the map $\alpha|_Y\otimes\beta|_Y$.
Suppose $\cS$ is a short exact sequence below,
$$0\rightarrow(H,E)\xrightarrow \alpha (K,F)\xrightarrow \beta (L,G) \rightarrow 0$$
and $(M,D)$ is a standard vector bundle. It is enough to prove the following sequence $D\otimes \cS$ is exact, the other being similar.
$$0\rightarrow \xi_{M\cap H}(D\otimes E)\xrightarrow{1\otimes \alpha} \xi_{M\cap K}(D\otimes F)\xrightarrow{1\otimes\beta} \xi_{M\cap L}(D\otimes G) \rightarrow 0$$
By the definition of exactness for standard vector bundles, we need to prove that the following sequence of $\OX$-modules is exact.
$$0\rightarrow \xi_{M\cap H}(D\otimes E)|_X\xrightarrow{(1\otimes \alpha)|_X} \xi_{M\cap K}(D\otimes F)|_X\xrightarrow{(1\otimes\beta)|_X} \xi_{M\cap L}(D\otimes G)|_X \rightarrow 0$$
But it is isomorphic to the sequence
$$0\rightarrow \xi_MD|_X\otimes \xi_HE|_X\xrightarrow{1\otimes \alpha|_X} \xi_MD|_X\otimes \xi_KF|_X\xrightarrow{1\otimes\beta|_X} \xi_MD|_X\otimes \xi_LG|_X \rightarrow 0,$$
which is an exact sequence of locally free $\OX$-modules since $\otimes$  is a biexact pairing on the category of locally free $\OX$-modules.

\end{proof}

\subsection{Twisted sheaf as a standard line bundle}\label{twshlb}
In this section, we discuss the twisted sheaf $\cO(n)$ on a projective space.
There could be many standard vector bundles that correspond to $\cO(n)$.
But there is a particular one that behaves well under pullbacks and base change.
In Theorem \ref{thmvb}, the way we constructed a standard vector bundle from an ordinary vector bundle was to use Corollary \ref{stdvb} after choosing a covering that trivializes the vector bundle and an isomorphism to a standard free module for each scheme factoring through one of the open covers. We will show how the choices can be made universally for $\cO(n)$.

Let $\PP^r_X = X\times_\ZZ\Proj \ZZ[x_0,x_1,\ldots,x_r].$
It is covered by $U_0,U_1,\ldots,U_r$ where
$$U_k=X\times_\ZZ\Spec \ZZ\left[\frac{x_0}{x_k},\frac{x_1}{x_k},
\ldots,\widehat{\frac{x_k}{x_k}},\ldots,\frac{x_r}{x_k}\right]
\quad \mbox{for}\quad k=0,1,\ldots, r. $$
Let $i_k:U_k\rightarrow \PP^r_X$ be the inclusions, and let $H$ be the sieve generated by them.
For each $k$, there is a map $x_k^n:\cO_{U_k}\rightarrow i_k^*\cO_{\PP^r_X}(n)$
of sheaves on $\zar{(U_k)}$ defined by multiplication by $x_k^n$. It is an isomorphism because $x_k$ is invertible in $U_k$.
Suppose $Y\xrightarrow h \PP^r_X$ is an object of $\cC_H$. Then $h$ factors as
$$Y\xrightarrow {h_k} U_k\xrightarrow {i_k} \PP^r_X$$
for some $k$. We choose the largest such $k$. (One could make many different choices here, but our choice is made for Proposition \ref{oninf} to work.) Then there is an isomorphism
$$\alpha_h:\cO_Y\cong h_k^*\cO_{U_k}\xrightarrow[\cong] {h_k^*(x_k^n)} h_k^*i_k^*\cO_{\PP^r_X}(n)\xrightarrow[\cong]{} h^*\cO_{\PP^r_X}(n) $$
Define $\varphi_h$ to be the following composite map as in Corollary \ref{stdvb}.
$$\bOY \cong B\OY \xrightarrow[\cong]{B\alpha_h} Bh^*\cO_{\PP^r_X}(n) \xrightarrow[\cong]{} h^*(B\cO_{\PP^r_X}(n)|_H).$$
With these choices of $H$ and $\varphi_h$'s, 
define $\cO_{\PP^r_X}(n)=S_H^\varphi B\cO_{\PP^r_X}(n)|_H$. 
This is a standard line bundle.

\begin{prop}\label{oninf}
Suppose $i_{\infty}:X\rightarrow \PP^1_X$ is the inclusion of the point $\infty=[0:1]$. Then $i_\infty^*\cO_{\PP^1_X}(n)=\cO_X$ for any $n$.
\end{prop}
\begin{proof}
The projective line $\PP^1_X$ is covered by two affine lines $U_0$ and $U_1$ as above.
Let $H$ be the sieve generated by them. For any scheme $Y$ over $X$, the composite map $Y\rightarrow X\xrightarrow {i_{\infty}} \PP^1_X$ factors through $U_1$.
Hence $i_\infty^*H=H_X$, and $i_\infty^*\cO_{\PP^1_X}(n)$ is defined uniformly by the isomorphism $x_1^n$ identifying $\cO_{U_1}$ with $i_1^*\cO_{\PP^1_X}(n)$. Therefore, $i_{\infty}^*\cO_{\PP^1_X}(n)=\cO_X$.
\end{proof}
\begin{prop}\label{onfunc}
Suppose $f:Y\rightarrow X$ be a map of schemes. Let $g:\PP_Y^r\rightarrow \PP_X^r$ be the induced map $f\times1$. Then $g^*\bO_{\PP_X^r}(n) = \bO_{\PP_Y^r}(n)$ (equality, not simply natural isomorphism) for every $n$.
\end{prop}
\begin{proof}
Let $H$ be the sieve on $\PP_X^r$ generated by the covering $\{U_{X,k}\rightarrow \PP_X^r\}$ described above,
and let $K$ be the sieve on $\PP_Y^r$ generated by the analogous covering
$\{U_{Y,k}\rightarrow \PP_Y^r\}$  of $\PP_Y^r$.
For each $0\le k\le r$, a map $h:Z\rightarrow \PP_Y^r$ factors through $U_{Y,k}$ if and only if the composite $gh$ factors through $U_{X,k}$ since $U_{Y,k}\cong U_{X,k}\times_{\PP_X^r}\PP_Y^r$.
$$\xymatrix{
Z\ar[dr] \ar[ddr]_h \ar@{-->}[drr] & & & Z\ar@{-->}[dr]\ar[ddr]_h \ar[drr] \\
& U_{Y,k} \ar[d] \ar[r] & U_{X,k}\ar[d]& & U_{Y,k} \ar[d] \ar[r] & U_{X,k}\ar[d]\\
& \PP_Y^r \ar[r]_g & \PP_X^r & & \PP_Y^r \ar[r]_g & \PP_X^r\\
}$$
Therefore, $g^*H=K$.
Moreover, for each object $h:Z\rightarrow \PP_Y^r$ of $\cC_K$,
the standardizing maps $\varphi_h$ for $\cO_{\PP_Y^r}(n)(Z)$ and $g^*\cO_{\PP_X^r}(n)(Z)$ are defined by multiplication by $x_k^n$, both with the same $k$. Therefore, $g^*\bO_{\PP_X^r}(n)=\bO_{\PP_Y^r}(n)$ as standard line bundles.
\end{proof}

\bibliographystyle{plain}
\bibliography{SVB}

\end{document}